\documentclass[11pt, a4paper]{article}

\newcommand{\bg}{\raisebox{1.7pt}{$\gamma$\hspace{-4.5pt}\raisebox{-6.5pt}{\textcolor{white}{$\blacksquare$}}} \hspace{-18.25pt} $\gamma$}

\def\bgamma{\textstyle{\textup{\bg}}}
\def\hgamma{\textstyle{\textup{$\gamma$}}}

\usepackage[utf8]{inputenc}
\usepackage[french,english]{babel}
\usepackage[T1]{fontenc}
\usepackage{amsmath}
\usepackage{amsfonts}
\usepackage{amssymb}
\usepackage{graphicx}
\usepackage{pdfpages}
\usepackage[left=3.5cm,right=3.5cm,top=3.5cm,bottom=3.5cm]{geometry}
\usepackage[all]{xy}
\usepackage{xcolor}
\definecolor{lien}{rgb}{0.7,0,0}
\definecolor{citation}{rgb}{0,0.9,0}
\usepackage[width=0.75\textwidth]{caption}
\usepackage{placeins}

\usepackage[cal=boondox]{mathalfa}

\usepackage{amsthm}
\theoremstyle{plain}
\newtheorem{thm}{Theorem}
\newtheorem{prop}[thm]{Proposition}
\newtheorem{lemme}[thm]{Lemma}
\newtheorem{coro}[thm]{Corollary}

\theoremstyle{definition}
\newtheorem{defn}[thm]{Definition}

\def\N{\mathbb{N}}
\def\R{\mathbb{R}}
\def\Z{\mathbb{Z}}
\def\C{\mathbb{C}}

\def\M{\mathbb{M}}
\def\E{\mathbb{E}}

\def\dis{\displaystyle}
\def\Sph{\mathbb{S}}
\def\dis{\displaystyle}

\def\rel{\mathcal{R}}
\def\im{\mathcal{I}}
\def\is{\mathcal{I \!\! s}}
\def\dj{\partial_j}

\def\d{\mathrm{d}}
\def\dist{dist}
\def\n{\mathrm{\textbf{n}}}
\def\t{\mathrm{\textbf{t}}}
\def\Lcal{\mathcal{L}}
\def\Rcal{\mathcal{R}}
\def\Mcal{\mathcal{M}}

\def\noi{\noindent}

\newcommand{\bCP}{CP_{\bgamma}}
\newcommand{\CP}{CP_{\hgamma}}
\newcommand{\s}{\mathfrak{S}}

\newcommand{\iconv}{\mathrm{IntConv}}

\title{Convex Integration Theory without Integration}
\date{September 5, 2019}
\author{Mélanie Theillière\footnote{\textsc{Université Claude Bernard Lyon 1, Institut Camille Jordan, Lyon, France}}}

\begin{document}

\maketitle

\begin{abstract}
We replace the usual Convex Integration formula by a {\it Corrugation Process} and introduce the notion of {\it Kuiper differential relations}. This notion provides a natural framework for the construction of solutions with self-similarity properties. We consider the case of the totally real relation, we prove that it is Kuiper and we state a totally real isometric embedding theorem. We then show that the totally real isometric embeddings obtained by the Corrugation Process exhibits a self-similarity property. Kuiper relations also enable a uniform expression of the Corrugation Process that no longer involves integrals. This expression generalizes the ansatz used in \cite{ContiEtAl} to generate isometric maps. We apply it to build a new explicit immersion of $\R P^2$ inside~$\R^3$.
\end{abstract}

\section*{General introduction}

\textbf{Context and motivation.--} Convex Integration is a theory developed by Gromov in the 70s and 80s to solve large families of differential relations and to describe the topology of the space of their solutions~\cite{Gromov-book}. It is inspired by the works of Nash and Smale~\cite{Nash54, Smale} and is commonly used in Riemannian, symplectic or contact geometry to determine the presence of a $h$-principle~\cite{Spring-book, EliaMisha-book, Geiges}. Currently, the theory is experiencing a renewed interest due to the discovery of its applicability in Fluid Mechanics~\cite{Lellis-Szeke,CdLdR,BV19} and in Contact Geometry~\cite{Murphy18}.\\

\noi
This article is motivated by the fractal behavior observed recently in the construction of two explicit $C^1$ isometric embeddings by using Convex Integration Theory: a square flat torus inside the Euclidean 3-space and a unit sphere inside a small ball~\cite{PNAS, BBDLRT}. These embeddings were obtained by stacking iteratively corrugations from an initial map.
In both cases, the resulting sequence of maps converges toward an isometric embedding whose differentials satisfy self-similarity properties. In~\cite{Gromov2017}, Gromov points out that such self-similarity properties could be used to possibly rigidify the $C^1$ theory of isometric maps. \\

\noi
\textbf{Three conditions for self-similarity.--} The self-similarity obtained for the torus or the sphere results from three specific choices that are maintained throughout the construction process:\\
\hspace*{3mm} $i)$ all the corrugations share the same pattern, \\
\hspace*{3mm} $ii)$ the directions of corrugation are periodically repeated,\\
\hspace*{3mm} $iii)$ unpleasant gluings between local constructions are avoided.\\
Since none of these points follows from the Convex Integration Theory, it is unclear that such construction choices could be made for arbitrary manifolds or arbitrary initial maps. This raises the question whether they could be included in the Convex Integration Theory so that to ensure self-similarity properties of the differential of generated maps. \\

\noi
\textbf{Substituting the Convex Integration formula.--}
Convex Integration Theory builds solutions of an order one differential relation $\rel$ by applying iteratively a Convex Integration process to a formal solution $\s$ of this relation. Each Convex Integration process generates a map which is defined via an integral formula whose integrand is expressed in terms of a loop family~$\gamma$ of $\rel$ chosen accordingly to $\s$. In this article, we propose to replace the Convex Integration formula by a new formula, that we call {\it Corrugation Process} (see Equation~\eqref{formula_CP}), which improves the previous one on two points: it is local and it preserves periodicity. These two properties allow to avoid many of the unpleasant gluings mentionned in point $iii)$.\\

\noi
\textbf{Generating similar corrugations.--} To deal with point $i)$, we then introduce the notion of Kuiper relations which, informally, are the differential relations that can be solved by loop families built on a pre-defined pattern. We prove that the relation of codimension one immersions and the relation of totally real immersions are Kuiper (see Theorem~\ref{th1}). We show that for Kuiper relations the contribution of the two main ingredients of the Corrugation Process -- namely the formal solution $\s$ and the loop family $\gamma$ -- can be separated and clearly interpreted. As a consequence, the geometric effect of a Corrugation Process can be precisely described: it generates corrugations whose shapes are induced by the pre-defined pattern. \\

\noi
\textbf{Convex Integration without integration.--} An unexpected outcome of this separation is that, in the case of a Kuiper relation, the Corrugation Process no longer involves any integration on $\s$ (see Proposition~\ref{prop_noInt}). This allows to build solutions with short and simple analytic expressions.   As an illustration, we build a new explicit immersion of $\R P^2$ inside $\R^3$ by a direct application of a Corrugation Process (see Section~\ref{sectionApplImmersion}).\\ 

\noi
\textbf{Self-similarities.--} The notion of Kuiper relation offers a natural framework for studying self-similarities that emerge in Convex Integration Theory. 
We first observe by considering the totally real relation that a kind of self-similarity can derive from points $i)$ and $iii)$. Precisely, we state a $C^1$ isometric immersion theorem for totally real map (see Theorem~\ref{th2}), first by proving that the totally real relation is Kuiper and then by using the Corrugation Process throughout the Nash-Kuiper iterative construction of an isometric map. We show that the Maslov component of its Gauss map bears some resemblance with a Weierstrass function (see Proposition~\ref{prop_maslov}).\\

\noi
A geometric self-similarity behavior is also revealed when we take into account point $ii)$. Indeed, the construction of a square flat torus inside the Euclidean 3-space or of a unit sphere inside a small ball can be done as in the articles~\cite{PNAS, BBDLRT} by replacing the Convex Integration formula by the Corrugation Process repeated in always the same set of three directions. As a consequence, the resulting isometric embeddings will share similar properties than the one described by the \emph{Riesz Asymptotic Behavior Theorem } of \cite{ensaios}: at every point of the embedding, the surface looks the same in the sense that the normal vector is asymptotically obtained as a infinite product of rotations associated to the same set of three directions (Proposition~\ref{prop:fin}).

\section*{Presentation of the results}

\noi
\textbf{Framework.--} In its traditional framework, Convex Integration considers sections $f\in \Gamma^{\infty}(X)$ of a fiber bundle $p:X\rightarrow M$ and encodes any differential constraint of order $r\geq 1$ as a subset $\rel\subset J^{r}(p)$ of the $r$-jet manifold of $p$. A section $\s$ of $p_r:J^r(p)\rightarrow M$ whose image lies inside $\rel$ is called a formal solution and its composition by the natural projection $p_{r,0}:J^r(p)\rightarrow X$ gives a section $f_0 = p_{r,0}\circ\s$ of $p$ called the base section. A solution of $\rel$ is a section of $p$ whose $r$-jet is a section of $p_r:\rel\rightarrow M$.\\ 

\noi
\textbf{Convex Integration Formula.--} In the case where $r=1$ and $\rel$ is both open and satisfies some convexity assumption, the theory formally builds solutions from a formal solution $\s$ of the differential relation $\rel$. This construction is carried out iteratively on charts diffeomorphic to the cube $[0,1]^m$, $m=\dim M$, and with boundary conditions. The solutions, as well as the modified formal solution, are then glued together following a process described in the proof of Theorem 4.2 of~\cite{Spring-book}. Over each cube, the base section of $\s$ reads as a map $f_0: [0,1]^m\rightarrow \R^n$  where $m$ and $n$ are the dimensions of $M$ and of the fiber of $p.$ A sequence of (at most) $m$ Convex Integrations is applied to modify $f_0$ iteratively on every coordinate direction $\partial_1,\ldots ,\partial_m$. Each of these Convex Integrations builds from the given map $f_0$ and a direction, for instance $\partial_1$, a new map $F_1$ defined by the following formula 
\begin{eqnarray}\label{eq_IC_formula}
F_1(x):= f_0(0,x_2,\ldots ,x_m)+ \int_{s=0}^{x_1} \gamma(s, x_2,\ldots,x_m , Ns) \ ds
\end{eqnarray}
where $x\in [0,1]^m$, $N\geq 1$ is any natural number and $\gamma:[0,1]^m\times\R/\Z\rightarrow \R^n$ is a family of loops. The new map $F_1$ has its $\partial_1$-derivative lying inside the image of $\gamma$. Since the ultimate goal is to solve $\rel$, it is required that this image is contained in some specific subspace depending on $\rel$ (see Lemma~\ref{gamma_dans_R}). An Average Constraint 
$$(AC) \hspace*{5mm} \int_0^1 \gamma(x,s)d s =\partial_1 f_0(x)$$
is also required to control the perturbation of the derivatives in the other directions. Finally, $\gamma$ is also chosen to be a round trip so that its image retracts onto its base point. This plays a crucial role when gluing or homotopying a local holonomic solution to match with the global formal solution. The existence and the construction of such a family of loops lay the foundations for Convex Integration Theory~(see~\cite{Spring-book}).\\

\noi
\textbf{The Corrugation Process.--} In this article, we replace the Convex Integration Formula by a {\it Corrugation Process} defined as follows. From a map $f_0:[0,1]^m\rightarrow\R^n$ we build a new map $f_1$ by setting
\begin{eqnarray*}
f_1(x):=f_0(x) + \frac{1}{N}\int_{s=0}^{Nx_1} \gamma(x,s)-\overline{\gamma}(x) ds 
\end{eqnarray*}
where $\overline{\gamma}(x)$ denotes the average of $t\mapsto \gamma(x,t).$  We then say that $f_1$ is obtained from $f_0$ by a Corrugation Process. This formula involves the same ingredients as the usual Convex Integration but space and time variables are separated and the integral relates only on the time variable. It ensues the following local property: the value of $f_1$ at $x$ only depends on the values of $f_0$ and of $\gamma$ at $x$. As a consequence, the map obtained by the Corrugation Process preserves the 1-periodicity, a property that allows a straightforward gluing on overlapping charts. In Section~\ref{section-CP}, we explore the properties of  the Corrugation Process and show that the whole Convex Integration Theory can be developed by using the Corrugation Process (see Proposition~\ref{propCP}). Note that, by a different approach, Eliashberg and Mishachev obtain a formula sharing the above local property~\cite{EliaMisha-book}. Nevertheless, their construction depends on a specific choice of paths (called {\it flowers}) and is less amenable to the uniform construction described in the next paragraphs.\\

\noi
\textbf{Kuiper relations.--} In the Convex Integration Theory, the family of loops $\gamma$ is built once the formal solution $\s$ is given. Indeed, three conditions are required on $\gamma$ and all of them involve $\s$: the average $\overline{\gamma}(x)$ should be some derivative of $f_0={\rm bs}\;\s$ at $x$, the base point of the loop $t\mapsto\gamma(x,t)$ should be $\s(x)$ and its image should lie inside a subset of $\rel$ depending on $\s(x)$. 
As a consequence, the exact contributions of the formal solution $\s$ and of the chosen family of loops $\gamma$ in the geometry of the resulting map is unclear. We propose to unravel the situation by constructing a larger family of loops~$\bgamma(\sigma,w)$ depending continuously on two parameters: its base point $\sigma\in \rel$ and its average $w.$ If such a construction is feasible, we say that the differential relation $\rel$ is {\it surrounding} (see Definition~\ref{def_surrounding}). In that case, the role of the family of loops and of the formal solution can be rigidly separated by defining $\gamma$ from the larger family~$\bgamma$:
$$\gamma(x,.):=\bgamma(\s(x),\partial_1 f_0(x)).$$ 
To have a complete understanding of $\gamma$ it then remains to determine the geometry of $\bgamma$. This can be done by building the surrounding family $\bgamma$ from a pre-defined loop pattern, i.e. a continuous family of loops $c:A\times\R/\Z\rightarrow\R^p$ with $A\subset\R^q$, which are used as models. We say that the surrounding family $\bgamma$ is {\it $c$-shaped} if every loop $t\mapsto \bgamma(\sigma,w)(t)$ is, in a uniform way, the linear image of a loop $t\mapsto c(a,t)$ from some $a={\bf a}(\sigma,w)$ (see Definition~\ref{def_gamma_c-shaped}). We call {\it Kuiper relation} a surrounding differential relation $\rel$ which admits such a $c$-shaped family (see Definition~\ref{def_R_Kuiper}).\\

\noi
\textbf{Eliminating integrals in the Corrugation Process formula.--} For a Kuiper relation, the analytic expression of the Corrugation Process does not include any integration of the formal solution. The only functions which are integrated are the coordinates $c_i$ of $c$ which are independent of $\s$. We denote by $C_i(a,.)$ the 1-periodic function defined by
$$t\mapsto C_i(a,t)=\int_0^tc_i(a,t)-\overline{c_i}(a)dt.$$  We show that the Corrugation Process writes
\begin{eqnarray*}
f_1(x) &=&  f_0(x) + \frac{1}{N} \sum_{i=1}^pC_i( a(x),Nx_1 )e_i(x)
\end{eqnarray*}
where $a(x) := \textbf{a}(\s(x), \partial_1 f_0(x))\in\R^q$, $e_i(x) := \textbf{e}_i(\s(x), \partial_1 f_0(x))\in\R^p$ for some maps ${\bf a}$ and~${\bf e}_i$ (see Proposition~\ref{prop_noInt}). This formula provides a clear interpretation of the Corrugation Process: the map $f_1$ is obtained from $f_0$ by adding perturbations along $p$ directions $e_1,\ldots,e_p$ which only depends on the initial formal solution $\s$. It matches with the integral free formulas used by Nash and Kuiper to iteratively modify a short embedding into a $C^1$-isometric embedding in~\cite{Nash54, Kuiper55}.  It also establishes the link between the Convex Integration process and formulas commonly found in Analysis (see Equation $(20)$ of~\cite{ContiEtAl} or the ansatz p.23 of~\cite{wasem} for instance). Mostly, it allows a direct approach of the Convex Integration Theory with a ready-made formula whose geometric effect, encoded by a single pattern, can be controlled (see Lemmas \ref{lemme_Jdensity1}, \ref{lemme_Jdensity2} or \ref{lemme_Maslov_angle} for examples of such a control).\\

\noi
\textbf{Main result.--} In \cite{LevyThurston}, Thurston introduces the Theory of Corrugation and uses it to recover many results in Immersion Theory including the Whitney-Graustein Theorem and some explicit constructions of sphere eversions. In this theory, the fundamental process is to perform normal deformations along curves by adding eight-shaped arcs. It is readily seen that these deformations coincide with a Corrugation Process (hence the name) with pattern an arc of parabola. It turns out that another simple pattern --an arc of circle-- allows to solve not only the immersion relation, but also the immersion of $\epsilon$-isometric maps and the totally real relation. Precisely, let $c:[0,\alpha_0]\times\R/\Z\rightarrow\R^3$ be the pattern defined by 
\begin{eqnarray}\label{pattern_arc_de_cercle}
c(\alpha, t)=\Big(\cos(\alpha \cos 2\pi t)- J_0(\alpha), \sin(\alpha \cos 2\pi t),1\Big)
\end{eqnarray}
with $J_0$ the Bessel function 
$J_0(\alpha) := \int_0^1 \cos(\alpha\cos 2\pi t)dt$
and $\alpha_0\approx 2.4$ its first positive root. The curve $t\mapsto c(\alpha,t)$ is an horizontal arc of circle of amplitude $2\alpha$ and of center $(-J_0(\alpha),0,1).$

\begin{thm}\label{th1}
The following relations are relative Kuiper relations with respect to $c$:
\begin{enumerate}
\item The immersion relation $\im(M,W)$ with $M$ and $W$ orientable and $\dim\; W=\dim\; M+1$. 
\item The totally real relation $\im_{TR}(M,W)$ where $(W,J)$ is an almost complex manifold with $\dim\; W=2\dim \;M$.
\end{enumerate}
\end{thm}

\noi
Recall that a map $f:M\rightarrow W$ is an immersion if for every $x\in M$,  $\mbox{rank}\;df_x$ is maximal and that $f:M\rightarrow (W,J)$ is totally real if $df(T_xM)\oplus Jdf(T_xM)=T_{f(x)}W$. In the particular case where $df(T_xM)$ is orthogonal to $Jdf(T_xM)$, the map is called Lagrangian. Totally real maps occur in Symplectic Geometry as a flexible version of Lagrangian maps and in the Theory of functions of several complex variables to build domains of holomorphy~\cite{Gromov-book}.\\

\noi
A direct consequence of this theorem is that these different relations are solved by using a similar formula. In the case $W=\R^n$, it writes:
\begin{eqnarray*}
f_1(x) = f_0(x) + \frac{1}{N}K_c(\alpha(x), Nx_1) \cdot e_1(x) + \frac{1}{N} K_s(\alpha(x), Nx_1) \cdot e_2(x)
\end{eqnarray*}
with
\begin{eqnarray*}
\begin{array}{lll}
  K_c(\alpha,t) & := & \dis\int_{u=0}^t \Big( \cos(\alpha \cos 2\pi u)-J_0(\alpha)\Big)du\vspace*{2mm}\\
  K_s(\alpha,t) & := & \dis\int_{u=0}^{t} \sin(\alpha \cos 2\pi u ) du
\end{array}
\end{eqnarray*}
and where the angle function $\alpha$ and the two vector fields $e_1$ and $e_2$ are defined in Section~\ref{proof_th1}. This formula is completely similar to the {\it ansatz} used by Conti, De Lellis and Székelyhidi to construct isometric embeddings and to study their $C^{1,\alpha}$-regularity~\cite{ContiEtAl}, \cite{wasem}. In that last case, the underlying differential relation is the one of $\epsilon$-isometric maps (see Section~\ref{Nash-Kuiper-historique}). It is quite remarkable that the same pattern allows to solve three distinct differential relations. We derive below two applications of the Main Theorem.\\

\begin{figure}[!ht]
\centering
\includegraphics[scale=0.22]{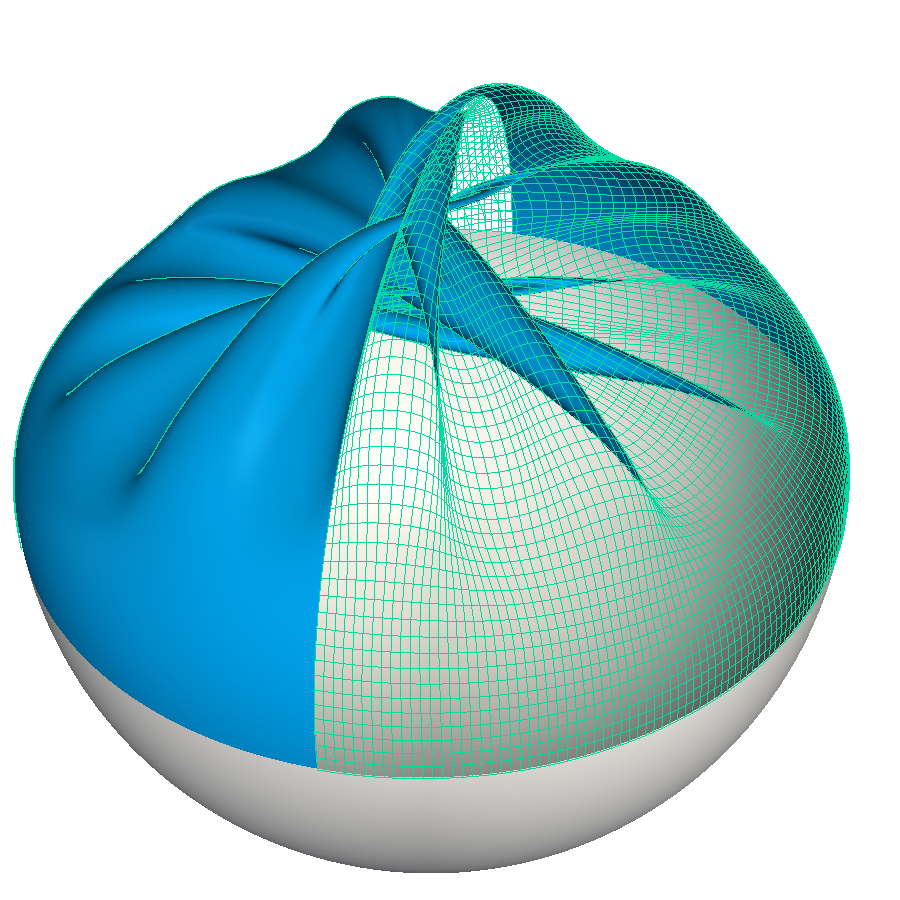}
\caption{ \footnotesize{ 
\textbf{An immersion of $\R P^2$ obtained by a Corrugation Process} (Image of the map $F_1$ of Subsection \protect \ref{RP2_close} with $\beta(x_1) =\dis \left(\frac{1+\cos(2\pi x_1/5)}{2}\right)^{0.75}$ and the same $\theta$, $\alpha$, $N$ as in Figure \protect \ref{image_conoid_local}.)}}\label{image_RP2}
\end{figure}

\noi
\textbf{Application 1: Immersion of $\R P^2$.--} The first immersion of $\R P^2$ into $\R^3$ is exhibited by Boy~\cite{Boy1901} in 1901, the now famous Boy's surface. Nevertheless, in his article the immersion is defined via a couple of drawings. The various attempts~\cite{Apery-book} to move from a pictorial to a formal description of this surface eventually find their way in the 1980's with the Apery parametrization~\cite{Apery86} and the Briant-Kusner immersion~\cite{Bryant84, Kusner87}. However, these results do not put an end to the search of noteworthy immersions of $\R P^2$ (see~\cite{GoodmanHoward} for a recent one). Here, we obtain a new immersion of $\R P^2$ by a direct application of a Corrugation Process on a neighborhood of the two singular points of a Plücker's conoid. The resulting maps define a family of immersions of $\R P^2$ whose analytic expressions are short and simple (see Equations~\eqref{equ_conoid_desing} and \eqref{equ_RP2_desing} of Subsections~\ref{corrugated_PC} and \ref{RP2_close}). The images of these immersions share some similarities with the {\it Tobacco Pouch Surfaces} of~\cite{Francis}.\\

\noi
\textbf{Application 2: Totally real isometric embeddings.--} The celebrated $C^1$-isometric embedding theorem of Nash \cite{Nash54} has been an inspirational source for the concept of $h$-principle and the development of the Convex Integration Theory by Gromov \cite{Gromov73, Gromov-book}. Several $C^1$-isometric theorems in the spirit of Nash have been obtained in many contexts, for Carnot-Caratheodory metrics~\cite{Dambra95}, in contact, symplectic and pseudo-Riemannian geometries~\cite{Dambra2000, DambraLoi2002, DambraDatta2006} and for sub-Riemannian manifolds~\cite{LeDonne2013}. \\

\noi
We endow $M$ with a Riemannian metric $g$ and we consider a (complete) almost Hermitian manifold $(W,J,h)$. A map $f_0:(M,g)\to (W,h)$ is strictly short if for all points of $M$, we have $f_0^*h\leq K g$ with $0\leq K<1$ (we denote by $h$ both the Hermitian metric of $W$ and the Riemannian metric it defines). The Nash-Kuiper Embedding Theorem states that if $f_0$ is strictly short then for every $\epsilon>0$ there exists a $C^1$-isometric map $f_{\infty}: (M,g)\rightarrow (W,h)$ such that $dist(f_{\infty}(x),f_0(x))\leq\epsilon$ for every $x\in M$. This map $f_{\infty}$ is obtained from $f_0$ by a sequence of deformations. Since $\im_{TR}$ is Kuiper with respect to $c$, we can replace each deformation by a Corrugation Process. In addition to produce a sequence $(f_{\ell})_{\ell\in\N}$ of totally real maps, the use of the Corrugation Process allows to control the geometry of each map to insure that the limit $f_{\infty}$ is totally real.

\begin{thm}\label{th2}
Let $(M^m,g)$ be a compact Riemannian manifold and $f_0:(M^m,g)\to (W^{2m},J,h)$ be a strictly short totally real immersion (resp. embedding). Then, for every $\epsilon>0$, there exists a $C^1$ totally real isometric immersion (resp. embedding) $f_{\infty}:(M^m,g)\to (W^{2m},J,h)$ such that $dist(f_{\infty}(x),f_0(x))\leq\epsilon$ for every $x\in M^m$.
\end{thm}

\noi
The proof crucially relies on the geometric control offered by the Kuiper property of the totally real relation.
We take advantage of the explicit expression~(\ref{pattern_arc_de_cercle}) of the pattern $c$ to provide in the following proposition a description of the Maslov component of the Gauss map of the map $f_{\infty}$ obtained from $f_0$ by the Corrugation Process (see Section~\ref{subsection_gauss_maslov} for the definition of the Maslov map).

\begin{prop}
\label{prop_maslov_intro}
Let  $\mathfrak{m}(f_0,f_{\infty})=e^{i\mathcal{W}_{\infty}}:M\rightarrow\Sph^1$ be the Maslov map of $f_{\infty}$ and $\mathcal{W}_{\infty}=2\sum_{\ell} \vartheta_{\ell}$ be the Maslov argument. Then if $\ell$ is large enough
$$\vartheta_{\ell}= \alpha_{\ell}\cos(2\pi N_{\ell}\pi_{\ell}) +  O\left( \frac{1}{N_{\ell}} \right).$$
\end{prop}

\noi
The angle $\alpha_{\ell}\cos(2\pi N_{\ell}\pi_{\ell})$ appearing in this proposition is the argument of the pattern $c$ used to build $f_{\ell}$ and with corrugation number $N_{\ell}$, amplitude $\alpha_{\ell}$ and direction $\pi_{\ell}$ depending on ${\ell}$. This proposition suggests that $\mathcal{W}_{\infty}$ shares some similarities with a Weiestrass function $x\mapsto \sum_k a^k\cos(2\pi b^kx)$, $x\in \R$. Recall that if $a<1<ab$, the Hausdorff dimension of the graph of this function is strictly larger than one and exhibits a self-similarity property.\\

\noi
In \cite{AhernRudin}, P. Ahern and W. Rudin provide a totally real embedding of $\Sph^3$ in $\C^3$. Any homothetic transformation with a sufficiently small factor of this embedding provides a strictly short totally real embedding.

\begin{coro}
There exist $C^1$ totally real isometric embeddings of the round 3-sphere inside $\C^3$.
\end{coro}

\noi
Note that the sphere $\Sph^m$ admits a totally real embedding in $\C^m$ if and only if $m=1$ or 3~\cite{Gromov-book}.
By comparison, there is no Lagrangian embedding of $\Sph^m$ inside $\C^m$ if $m>1$~\cite{Gromov85}.

\section*{Acknowledgements}

The results presented here are a part of my PhD Thesis. I wish to express my gratitude towards the whole team of the {\it Hevea project} and specially to my advisors Vincent Borrelli and Boris Thibert for their help and support. I greatly thank Patrick Massot for his interest towards this work and many valuable comments that help me to improve the content of this text.

\tableofcontents

\section{Corrugation Process}\label{section-CP}
 
\subsection{Convex Integration}\label{defCI}

We first recall the Convex Integration formula. Let $f_0:[0,1]^m\rightarrow \mathbb{R}^n$ be a map of class $C^k$ with $k \geq 1$, $\dj$ be a direction, $\gamma :[0,1]^m\times\R/\Z \rightarrow \mathbb{R}^n$ be a loop family of class $C^{k-1}$ and $N\in \; ]0,+\infty[$ be a number. The Convex Integration process builds a map
$F_1:[0,1]^m\rightarrow \mathbb{R}^n$ of class $C^{k-1}$ by setting :
\begin{eqnarray}\label{formula_CI}
F_1(x):= f_0(x_1,\ldots ,0,\ldots ,x_m)+ \int_{s=0}^{x_j} \gamma(x_1,\ldots ,s,\ldots ,x_m,Ns) \ ds.
\end{eqnarray}
We say that $F_1$ is obtained from $f_0$ by {\it Convex Integration} in the direction $\partial_j$ and write $F_1=CI_{\hgamma}(f_0,\dj,N)$.\\

\noi
The space $[0,1]^m$ is foliated by the line segments $S_{x_{\widehat{\jmath}}}:=(x_1,\ldots,x_{j-1})\times [0,1] \times (x_{j+1},\ldots,x_m)$, defined for each $x_{\widehat{\jmath}}:=(x_1,\ldots,x_{j-1},x_{j+1},\ldots,x_m)\in [0,1]^{m-1}$. In the formula, the integration is performed along these segments between the boundaries $x_j=0$ and $x_j=1$.

\begin{defn}\label{def_average_cond}
We say that a loop family $\gamma :[0,1]^m\times\R/\Z\rightarrow \mathbb{R}^n$ satisfies the \textit{Average Constraint with respect to $f_0$ in the direction $\partial_j$} if
\begin{eqnarray*}
(AC) \ \ \ \ \forall x\in [0,1]^m,\ \ \ \partial_{j} f_0(x) = \int_0^1 \gamma(x,t)dt.
\end{eqnarray*}
\end{defn}

\noi
Remark that the Average Constraint implies that $x\mapsto \gamma(x,t)$ has the same regularity than $x\mapsto \partial_{j} f_0(x)$. In practice, the loop family $\gamma$ is assumed to be $C^{k-1}$ and the map $f_0$ is $C^k$.

\begin{prop}{(\cite{ensaios})}\label{propICFormula}
Let $f_0:[0,1]^m\rightarrow \mathbb{R}^n$ be a map of class $C^k$ with $k\geq 2$ and $\gamma:[0,1]^m\times\R/\Z\rightarrow \mathbb{R}^n$ be a loop family of class $C^{k-1}$ which satisfies the average constraint (AC). Then the map $F_1=CI_{\hgamma}(f_0,\dj,N)$ is of class $C^{k-1}$ and satisfies 
\begin{itemize}
\item[$(P_1)$]\ \ $\|f_0-F_1\|_{\infty} = O(1/N)$, ($C^0$-density),
\item[$(P_2)$]\ \ $\| \partial_{i} f_0 -\partial_{i} F_1 \|_{\infty} = O(1/N)$ for every $i\neq j$,
\item[$(P_3)$]\ $\forall x\in [0,1]^m,\ \ \ \partial_{j} F_1(x)=\gamma(x,Nx_j)$.
\end{itemize}
\end{prop}

\noi
Note that property $(P_3)$ follows directly from the Convex Integration formula and that the constants involved in the notation $O(1/N)$ can be taken to be 
$C_{P_1}=2\|\gamma\|_{\infty}+\|\dj\gamma\|_{\infty}$ and $C_{P_2}=2\|\partial_i\gamma\|_{\infty}+\|\partial_i\dj\gamma\|_{\infty}$ (see \cite{PNAS}).

\subsection{Corrugation Process and comparison with Convex Integration Formula}\label{subsectionCP}

\begin{defn}\label{def1}
Let $f_0:[0,1]^m\rightarrow \mathbb{R}^n$ be a map, $\dj$ be a direction, $\gamma :[0,1]^m\times\R/\Z \rightarrow \mathbb{R}^n$ be a loop family and $N\in\; ]0,+\infty[$. We define the map $f_1:[0,1]^m\rightarrow \mathbb{R}^n$ by 
\begin{eqnarray}\label{formula_CP}
f_1(x):=f_0(x) + \frac{1}{N}\Gamma(x,Nx_j)
\;\;\mbox{ with }\;\;
\Gamma(x,t) = \int_{s=0}^t \gamma(x,s)-\overline{\gamma}(x) ds 
\end{eqnarray}
and $\overline{\gamma}(x)=\int_0^1 \gamma(x,s) ds$ denotes the average of the loop $t\mapsto\gamma(x,t)$. 
\end{defn}

\noi
\textbf{Remark.} If $x\mapsto\gamma(x,\cdot)$ is a $C^{k}$-map and $t\mapsto \gamma(\cdot,t)$ is a $C^{k-1}$-map, then $(x,t)\mapsto \Gamma(x,t)$ is a $C^{k}$-map. 
\\

\noi
We say that $f_1$ is obtained from $f_0$ by a {\it Corrugation Process} in the direction $\partial_j$ and we denote $f_1=CP_{\hgamma}(f_0,\dj,N)$. The real number $N$ is called the {\it number of corrugations}.

\begin{prop}\label{propCP}
Let $f_0:[0,1]^m\rightarrow \mathbb{R}^n$ be a map of class $C^k$ with $k\geq 2$ and $\gamma:[0,1]^m\times\R/\Z\rightarrow \mathbb{R}^n$ be a loop family such that $x\mapsto\gamma(x,\cdot)$ is $C^{k-1}$ and $t\mapsto \gamma(\cdot,t)$ is $C^{k-2}$. Then the map $f_1=CP_{\hgamma}(f_0,\dj,N)$ is of class $C^{k-1}$ and we have
\begin{itemize}
\item[$(P_1)$] \ \ $\|f_0-f_1\|_{\infty} = O(1/N)$ ($C^0$-density).
\item[$(P_2)$] \ \ $\| \partial_{i} f_0 -\partial_{i} f_1 \|_{\infty} = O(1/N)$ for every $i\neq j$.
\item[$(P_3')$]\ \ If $\gamma$ satisfies the Average Constraint (AC) with respect to $f_0$ in the direction~$\partial_j$ then $\partial_{j} f_1(x) = \gamma(x,Nx_j) + O(1/N)$ for all $x\in [0,1]^m$.
\end{itemize}
\end{prop}

\noi
{\bf Remark 1.-- } Propositions \ref{propICFormula} and \ref{propCP} show that both the Corrugation Process and the Convex Integration Formula share the same property $P_1$, $P_2$ and $P_3$ up to an $O(1/N)$. It ensues that the two formulas can be used equally in the Convex Integration Theory.\\ 

\noi
{\bf Remark 2.-- } From the proof below, it is readily seen that the constants involved in the notation $O(1/N)$ can be taken to be $K_{P_1}=2\|\gamma\|_{\infty}$, $K_{P_2}=2\|\partial_i\gamma\|_{\infty}$ and $K_{P_3'}=2\|\partial_j\gamma\|_{\infty}.$\\

\begin{proof}
Property $(P_1)$ is obvious since $\Gamma$ is a periodic map, so bounded. Regarding the derivatives in the direction $\partial_i\neq \partial_j$, since $\Gamma$ is $C^1$, we have 
\begin{eqnarray*}
\partial_{i} f_1(x)
&=& \partial_{i} f_0(x) + \frac{1}{N}\partial_{i}\left( \Gamma(x,Nx_j)\right)\\
&=& \partial_{i} f_0(x) + O\left( \frac{1}{N}\right).
\end{eqnarray*}
and this shows property $(P_2).$ For the property $(P_3')$, since $\gamma$ and $\overline{\gamma}$ are $C^1$, we have:
\begin{eqnarray*}
\partial_{j} f_1(x)
&=& \partial_{j} f_0(x) + \frac{1}{N} \partial_ {j} \left( \int_{t=0}^{Nx_j} \gamma(x,t)- \overline{\gamma}(x)dt \right) \\
&=& \partial_{j} f_0(x) + \gamma(x,Nx_j)- \overline{\gamma}(x) +\frac{1}{N} \int_{t=0}^{Nx_j} \partial_{j} \Big( \gamma(x,t)-\overline{\gamma}(x)\Big)dt \\
&=& \partial_{j} f_0(x) + \gamma(x,Nx_j)- \overline{\gamma}(x) +O\left(\frac{1}{N} \right).
\end{eqnarray*}
Since $\overline{\gamma}(x) -\partial_{j} f_0(x)=0$ by the Average Constraint, we obtain $(P_3').$ 
\end{proof}

\subsection{Extra properties}\label{extraProp}

\noi
\textbf{Coordinate-free expression:} The formula~(\ref{formula_CP}) of the Corrugation Process can be rewritten without using coordinates as follows
$$f_1=\CP(f_0,\pi,N):x\mapsto \exp_{f_0(x)}\frac{1}{N}\int_0^{N\pi(x)}(\gamma(x,t)-\overline{\gamma}(x))dt.$$
In that formula, $f_0:U\rightarrow (W,h)$ is a map from an open set $U\subset M$ to a complete manifold $W$ endowed with a Riemannian metric $h$, $\exp$ is the exponential map induced by $h$, $\pi:U\rightarrow \R$ is a submersion and $\gamma:U\times \R/\Z\rightarrow f_0^*TW$ is a $C^{k-1}$ loop family such that $\gamma(x,.):\R/\Z\rightarrow f_0^*TW_x$ for every $x\in U$. This expression reduces to the one in Definition~\ref{def1} if $M=[0,1]^m$, $W=\E^n$ and $\pi(x)=x_j$. In that case $d\pi=dx_j=\langle \partial_j,\cdot\rangle$ and $\exp_{f_0(x)}y = f_0(x) +y$ for every $x\in [0,1]^m$ and $y\in \R^n$. \\

\noi
Properties $(P_1)$, $(P_2)$ and $(P_3')$ may also be reformulated. If $d_W$ and $d_{TW}$ are arbitrary distances in $W$ and $TW$, we have
\begin{itemize}
\item[$(P_1)$] $d_W(f_1(x),f_0(x))=O(1/N)$ for every $x\in U,$
\item[$(P_2)$] $d_{TW}((df_1)_x(v),(df_0)_x(v))=O(1/N)$ for every $v\in\ker d\pi_x.$
\item[$(P_3')$] $d_{TW}\Big((df_1)_x(u), \gamma(x,N\pi(x))\Big) = O(1/N)$ for every $u\in T_xM$ such that $d \pi_x(u)=1$ and if $\gamma$ satisfies the Average Constraint $\overline{\gamma}(x)=(df_0)_x(u).$
\end{itemize}
The proofs are left to the reader.\\

\noi
\textbf{Relative property:} If $t\mapsto\gamma(x,t)$ is constant then $f_1(x)=\CP(f_0,\pi,N)(x)$ is $f_0(x)$.\\ 

\noi
\textbf{Periodicity property:} Let $f_1=\CP(f_0,\pi,N)$. If $N$ is a non-zero integer and $x\in U$ is such that $\pi(x)\in\Z$ then $f_1(x)=f_0(x)$. Indeed, the map $t\mapsto \Gamma(x,t)$ is $1$-periodic.

\section{Differential relations}\label{sectionDiffRelation}

Convex Integration is a powerful tool to solve differential equations/inequations of order one. These differential conditions define a subspace $\rel$ of the 1-jet space called {\it a differential relation}. Under some convexity conditions on $\rel$, the Convex Integration formula produces from any appropriate initial map a solution by applying Formula~(\ref{formula_CI}) iteratively. Similarly, the Corrugation Process Formula~(\ref{formula_CP}) allows to build solutions of a differential relation provided that the same convexity conditions hold. In this section, we introduce the notion of a loop pattern, that is a family of loops $c$ indexed by a finite number of parameters (see Definition~\ref{def_c}). We address the question of finding universal loop families $\bgamma$ independently of the data of the initial map (see~\ref{surrounding_dr}). We say that $\rel$ is a Kuiper relation with respect to $c$ if a $c$-shaped family $\bgamma$ exists (see~\ref{Kuiper-rel}). We then show that, for a Kuiper relation, the integrand in the Corrugation Process formula only involves the map $c$.

\subsection{Solutions and subsolutions of a differential relation}

\noi
We denote by $J^1(M,W)\rightarrow M$ the $1$-jet bundle of $C^1$-maps from $M$ to $W$. A point $\sigma$ in  $J^1(M,W)$ is a triple $\sigma=(x,y,L)$ where $x\in M$, $y\in W$ and $L\in\mathcal{L}(T_xM,T_yW)$. We endow the $1$-jet space with a distance function $\dist$. A \textit{differential relation of order $1$} is a subset $\rel\subset J^1(M,W)$, a section $\s:M\rightarrow \rel$ is called a \textit{formal solution}. It is called \textit{holonomic} if $\s=j^1f$  for some map $f:M\rightarrow W.$\\

\noi
Given a differential relation, our goal is to use the Corrugation Process to build from a section $\s$ a solution $f$ of $\rel.$ To do so, we have to add some additional assumption on the base map $f_0=bs\;\s$ of the section $\s$. 
Let $\sigma=(x,y,L)\in \rel$ and $(\lambda, u)\in T_x^*M\times T_xM$ such that $\lambda(u)=1$. We set
$$\rel(\sigma,\lambda,u) := Conn_{L(u)}\{ v\in T_yW \; |\; (x,y,L+(v-L(u))\otimes\lambda)\in \rel\}$$
where $Conn_aA$ denotes the path connected component of $A$ that contains $a$. Note that the linear map $L+(v-L(u))\otimes\lambda$ coincides with $L$ over $\ker \lambda$ and maps $u$ to $v$. We then denote by $\iconv\;\rel(\sigma,\lambda,u)$ the convex hull of $\rel(\sigma,\lambda,u_x).$

\begin{defn}
Let $U\subset M$, $\pi:U\rightarrow \R$ be a submersion and $u:U\rightarrow TM$ be a vector field such that $d\pi_x(u_x)=1$. Let $x\mapsto\s(x)=(x, f_0(x),L(x))$ be a formal solution of $\rel$ over $U$ with base map $f_0=\mathrm{bs}\;\s$. If for all $x$ in $U$ the map $f_0$ satisfies
$$df_0(u_x)\in \iconv\;\rel(\s(x),d\pi_x,u_x),$$
then the formal solution $\s$ is called a {\it subsolution of} $\rel$ {\it with respect to} $(d\pi,u)$.
\end{defn}

\noi
The following lemma is a straightforward consequence of the relative property of the Corrugation Process (cf Subsection~\ref{extraProp}).

\begin{lemme}\label{gamma_dans_R}
Let $\rel$ be an open differential relation and let $\s$ be a subsolution of $\rel$ with respect to $(d\pi,u)$ and with base map $f_0= \mathrm{bs}\;\; \s$. If $\gamma$ satisfies the Average Constraint with respect to $f_0$ in the direction $u$ and if 
$$\forall (x,t)\in U\times\R/\Z,\;\;\;\;\;\gamma(x,t)\in  \rel(\s(x),d\pi_x,u_x)$$
then, for $N$ large enough, $f_1:=\CP(f_0,\pi,N)$ satisfies
$$\forall x\in U,\;\;\;\;\;df_1(u_x)\in \mathcal{R}(\mathfrak{S}(x),d\pi_x,u_x).$$
\end{lemme}

\subsection{Surrounding differential relations}
\label{surrounding_dr}

Given a subsolution of a differential relation $\rel$ with respect to $(d\pi,u)$, the Corrugation Process builds a solution of $\rel$ in the direction $u$ providing that the loop family $\gamma$ satisfies the conditions of Lemma~\ref{gamma_dans_R}. In general finding such a loop family  is an issue which depends on the topology and the geometry of $\rel$, on the direction $u$ and on the point under consideration. In this section, we distinguish a class of differential relations (see Definition~\ref{def_surrounding}) for which the existence of such a family is uniformly given by a map $\bgamma$ depending on a base point $\sigma$ in $\rel$ and a barycentric point $w$ in $\iconv\;\rel(\sigma,d\pi_x,u_x)$.\\

\noi
We consider the bundle  $p_y^*TW$ over $\rel$ induced by the projection $p_y:\rel\rightarrow W$, $\sigma=(x,y,L)\mapsto y,$ and we define
$$\iconv(\mathcal{R},d\pi,u):=\{(\sigma,w)\in p_y^*TW\; |\; w\in \iconv\;\rel(\sigma,d\pi_x,u_x)\}.$$

\begin{defn}\label{def_surrounding}
Let $k\geq 1$ and $\mathcal{R}$ be a differential relation of $J^1(U,W)$. We say that a loop family
\begin{eqnarray*}
\begin{array}{lclc}
\bgamma : & \iconv(\mathcal{R},d\pi,u) & \stackrel{C^{k}}{\longrightarrow} & C^{k-1}(\R/\Z,TW)\\
& (\sigma,w) & \longmapsto & \bgamma (\sigma,w)(\cdot)
\end{array}
\end{eqnarray*}
is {\it surrounding with respect to} $(d\pi,u)$ if for every $(\sigma,w)$ we have
\begin{itemize}
\item[$(1)$]\ \ $t\mapsto \bgamma (\sigma,w)(t)$ is a loop in $\mathcal{R}(\sigma,d\pi_x,u_x)$,
\item[$(2)$]\ \ the average of $t\mapsto \bgamma (\sigma,w)(t)$ is $w$,
\item[$(3)$]\ \ there exists a continuous homotopy $H:\iconv(\mathcal{R},d\pi,u)\times [0,1]\rightarrow TW$ such that
$H(\sigma,w,0)=\bgamma (\sigma,w)(0)$, $H(\sigma,w,1)=L(u_x)$ and $H(\sigma,w,t)\in \mathcal{R}(\sigma,d\pi_x,u_x)$ for all $t\in\, [0,1].$
\end{itemize} 
\end{defn}

\noi
The third point insures the homotopic properties needed to glue local solutions and to state a potential $h$-principle for $\rel$. \\

\noi
Let $\rel$ be surrounding with respect to $(d\pi,u)$  and $\s$ be a $C^{k-1}$ subsolution of $\mathcal{R}$ with base map $f_0 =\mathrm{bs}\;\s$ of class $C^k$ ($k\geq 2$). Let $\bgamma :\iconv(\mathcal{R},d\pi,u)\rightarrow  C^{k-1}(\R/\Z,TW)$ be a surrounding loop family.
We define a loop family $\gamma$ by
$$ \gamma(x,t):=\bgamma(\s(x), df_0(u_x))(t)\in \rel(\sigma,d\pi_x,u_x)$$
for every $(x,t)\in U\times\R/\Z$. Note that $x\mapsto\gamma(x,\cdot)$ is $C^{k-1}$ and $t\mapsto\gamma(\cdot,t)$ is $C^{k-2}$. It ensues that the Corrugation Process $CP_{\hgamma}(f_0,\pi,N)$ is a  $C^{k-1}$ solution of $\rel$ in the direction $u$.\\

\noi
\textbf{Notation.--} If $\rel$ is surrounding with respect to $(d\pi,u)$ we denote by $CP_{\bgamma}(\s,\pi,N)$ the Corrugation Process $CP_{\hgamma}(f_0,\pi,N)$ given by the previous construction. Observe that if $\s$ and $\s'$ coincide on $x\in U$ then $CP_{\bgamma}(\s,\pi,N)(x)=CP_{\bgamma}(\s',\pi,N)(x).$

\begin{defn}\label{def_relative}
Let $\delta$>0 and $\mathcal{R}$ be a surrounding differential relation with respect to $(d\pi,u)$. We say that a loop family $\bgamma_{\delta} : \iconv(\rel,d\pi,u) \rightarrow C^k(\R/\Z,TW)$ is $\delta$-\textit{relative with respect to $(d\pi,u)$} for every $(\sigma,w) \in\iconv(\rel,d\pi,u)$ we have
\begin{eqnarray*}
L(u)=w \;\mbox{ and }\;\dist(w, \rel(\sigma,d\pi_x,u_x)^C)\geq\delta\;\;\;\Rightarrow \;\;\;\bgamma_{\delta}(\sigma,w)\equiv w,
\end{eqnarray*}
where $\rel(\sigma,d\pi_x,u_x)^C$ denotes the complement of $\rel(\sigma,d\pi_x,u_x)$ in $T_y W$. If for every $\delta\in\; ]0,\delta_0]$ (for some $\delta_0>0)$ there exists a $\delta$-relative loop family $\bgamma_{\delta}$ we say that $\bgamma:\,]0,\delta_0]\times \iconv(\rel,d\pi,u) \rightarrow C^k(\R/\Z,TW)$ is {\it relative}.
\end{defn}

\noi
By convention the distance function to an empty set is infinite, so if $\rel(\sigma,d\pi_x,u_x)^C=\emptyset$, then the above condition on $\dist(w, \rel(\sigma,d\pi_x,u_x)^C)$ is fulfilled.\\

\noi
Note that if $\delta$ is greater than the diameter of $\mathcal{R}$ then the $\delta$-relativity condition is empty. Therefore, if this diameter is not infinite, the definition is meaningful only when $\delta$ is sufficiently small.\\

\noi
If $\dist(w, \rel(\sigma,d\pi_x,u_x)^C) >0$, then the point $w$ lies inside the differential relation and is the average of the constant loop $\bgamma(\sigma,w)\equiv w$. The constant loops condition in the above definition allows to build, from a subsolution, a holonomic section such that the two base maps coincide outside a $\delta$-neighborhood of $\rel(\sigma,d\pi_x,u_x)^C$.  The following lemma is a direct consequence of the relative property of the Corrugation Process (cf. Subsection~\ref{extraProp}).

\begin{lemme}\label{lemme-relatif}
Let $\mathcal{R}$ be a surrounding open differential relation with respect to $(d\pi,u)$ and $x\mapsto \s(x)$ be a subsolution of $\rel$ for $(d\pi,u)$. Let $\bgamma$ be a $\delta$-relative loop family with respect to $(d\pi,u)$ and $f_0= \mathrm{bs}\;\s$, $f_1=CP_{\bgamma}(\s,\pi,N)$. Then 
$$ f_1(x) =f_0(x)$$
for every $ x \in U$  such that $\dist(df_0(u_x), \rel(\sigma,d\pi_x,u_x)^C) \geq \delta$.
\end{lemme}

\subsection{Shaped loop families and Kuiper relations}\label{Kuiper-rel}

\noi
The {\it Integral Representation Theorem} of Convex Integration Theory states that, given a subsolution $\s:[0,1]^m\rightarrow\rel$ and a direction $\partial_j$, it is possible to build a loop family $\gamma:[0,1]^m\times\R/\Z\rightarrow\rel$ such that, for each point $x\in\;[0,1]^m$, the average of $t\mapsto\gamma(x,t)$ is $\dj f_0(x).$ This Representation Theorem is a key ingredient since the Convex Integration formula lies on it (see~\ref{defCI}). In a surrounding relation, this approach can be reversed. The family of loops $\bgamma$ is defined on the far larger space $\iconv(\mathcal{R},d\pi,u)$ and uniformly generates a $\gamma$ for every subsolution $\s$ with respect to $(d\pi,u)$. This procedure allows to simplify the expression of the Corrugation Process if the family of loops $\bgamma$ is $c$-shaped, i.e. built from a single closed curve $c:\R/\Z\rightarrow\R^p$ seen as a pattern.

\begin{defn}\label{def_c}
Let $p,q>0$ be two natural numbers and $A\subset \R^q$ be a parameter space.
A {\it loop pattern} is a map $c:A\times \R /\Z\rightarrow\R^p$ such that $a\mapsto c(a,\cdot)$ is $C^{k-1}$ and $t\mapsto c(\cdot,t)$ is $C^{k-2}$ for some $k\geq 2$.
\end{defn}

\noindent
\textbf{Notation.--}
If $c$ is a loop pattern, we set
\begin{eqnarray*}
C(a,t):=\int_{u=0}^t \Big( c(a,u)-\overline{c}(a) \Big) du.
\end{eqnarray*}
As the loop $t\mapsto c(a,t)$ is $1$-periodic and $\overline{c}$ its average, the map $t\mapsto C(a,t)$ is $1$-periodic.
We denote by $(C_1(a,\cdot), \ldots,C_{p}(a,\cdot))$ the components of $C(a,\cdot)$.\\

\noi
We denote by $E\rightarrow W$ the fiber bundle over $W$ with fiber $\mathcal{L}(\R^p,T_yW)=(\R^p)^*\otimes T_yW$ and we consider its pull back by the projection $q:\iconv(\rel,d\pi,u)\rightarrow W$, $(\sigma,w)\mapsto y$. A section $\mathcal{e}$ of $q^*E$ defines a family of linear maps $\mathcal{e}(\sigma,w):\R^p\rightarrow T_yW$.

\begin{defn}\label{def_gamma_c-shaped}
Let $\mathcal{R}$ be a surrounding differential relation with respect to $(d\pi,u)$ and let $c:A\times \R / \Z \rightarrow \R^p$ be a loop pattern for some $k\geq 2$. A surrounding loop family $\bgamma: \iconv(\mathcal{R},d\pi,u) \rightarrow C^{k-2}(\R/\Z,TW)$ is said to be \textit{$c$-shaped} if there exist a section $\mathcal{
e}$ of $q^*E\rightarrow \iconv(\mathcal{R},d\pi,u)$ and a map
${\bf a}:\iconv(\mathcal{R},d\pi,u)\rightarrow A$ such that 
\begin{eqnarray*}
\bgamma(\sigma,w)(t) = \dis \mathcal{
e}(\sigma,w)\circ c({\bf a}(\sigma,w),t)
\end{eqnarray*}
for all $((\sigma,w),t)\in \iconv(\mathcal{R},d\pi,u)\times \R/\Z.$\\

\noi
\textbf{Notation.--} If $(c_1,\ldots,c_p)$ denote the components of $c$ in the standard basis of $\R^p$ and if  ${\bf e}_1, \ldots, {\bf e}_p $ denote the image of this basis by $\mathcal{e}$, we write
\begin{eqnarray*}
\bgamma(\sigma,w)(t) =c({\bf a}(\sigma,w),t) \cdot {\bf e}(\sigma,w)= \dis \sum_{i=1}^p c_i({\bf a}(\sigma,w),t)\, {\bf e}_i(\sigma,w).
\end{eqnarray*}
\end{defn}

\begin{defn}\label{def_R_Kuiper} Let $c$ be a loop pattern. We say that $\mathcal{R}$ is a {\it Kuiper relation with respect to $(c,d\pi,u)$} if there exists a surrounding loop family $\bgamma$ which is $c$-shaped. Moreover, if $\bgamma$ is relative we say that $\mathcal{R}$ is a {\it relative Kuiper relation}.
\end{defn}

\subsection{Corrugation Process with no integration}

\begin{prop}\label{prop_noInt}
Let $c$ be a loop pattern, $\mathcal{R}$ be an open Kuiper relation with respect to $(c,d\pi,u)$, $\mathfrak{S}=(x,f_0,L_0)$ be a subsolution and $\bgamma$ be a $c$-shaped surrounding loop family. Then $f_1 = CP_{\bgamma}(\mathfrak{S},\pi,N)$ has the following analytic expression
\begin{eqnarray*}
f_1(x) &=&  \exp_{f_0(x)}\left(\frac{1}{N} C( a(x),N\pi(x) )\cdot e(x)\right)
\end{eqnarray*}
where $a(x) := \textbf{a}(\s(x), df_0(u_x))$, $e(x) := \textbf{e}(\s(x), df_0(u_x))$ and $x\in U$. Moreover, if $N$ is large enough, the section
$$
x\mapsto \s_1:=(x,f_1,L_1 = L_0+(df_1(u_x)-L_0(u_x))\otimes d\pi)
$$
is a section in $\rel$. If $\rel$ is $\delta$-relative then
$\s_1(x) =\s_0(x)$
for every $x \in U$  such that $\dist(df_0(u_x), \rel(\sigma,d\pi_x,u_x)^C) \geq \delta$.
\end{prop}

\begin{proof}
By the non-coordinate expression of the Corrugation Process of Subsection~\ref{extraProp}, we have
\begin{eqnarray*}
\forall x\in U, \,\,\, f_1(x) = \exp_{f_0(x)}\frac{1}{N}\int_0^{N\pi(x)}(\gamma(x,t)-\overline{\gamma}(x))dt
\end{eqnarray*}
The loop family $\gamma$ is given by $\gamma(x,t):= \bgamma (\s(x),d f_0(u_x)) (t)$  and since the map $\bgamma$ is $c$-shaped and surrounding, we have
\begin{eqnarray*}
\gamma(x,t) := c(a(x),t)\cdot e(x)
\end{eqnarray*}
with $\overline{\gamma}(x) = df_0(u_x)$. Then
\begin{eqnarray*}
\int_{s=0}^t \gamma(x,s) -\overline{\gamma}(x)ds
= \left( \int_{s=0}^t c(a(x),s) - \overline{c}(a(x))ds \right) \cdot e(x)= C(a(x),t) \cdot e(x)
\end{eqnarray*}
It ensues that, for all $x\in U$, we have
$$f_1(x)=\exp_{f_0(x)}\left(\frac{1}{N} C( a(x),N\pi(x) )\cdot e(x)\right).$$
According to property $(P_3')$ of proposition~\ref{propCP}, for every $x\in U$, we have 
$$d_{TW}\Big((df_1)_x(u), \gamma(x,N\pi(x))\Big) = O(1/N).$$
Since $\rel$ is open and $\gamma(\cdot,\cdot)\in\rel(\s(x),d\pi_x, u_x)$, if $N$ is large enough, we have
$$x\mapsto (x,f_0(x),L_0+(df_1(u_x)-L_0(u_x))\otimes d\pi)\in \rel$$
for all $x\in U$. According to property $(P_1)$ of the same proposition, we also have
$$x\mapsto \s_1(x)=(x,f_1(x),L_0+(df_1(u_x)-L_0(u_x))\otimes d\pi)\in \rel.$$
If $\mathcal{R}$ is relative then the equality $f_1(x)=f_0(x)$ holds for all $x$ such that $\dist(df_0(u_x), \rel(\sigma,d\pi_x,u_x)^C) \geq \delta$ (cf Lemma~\ref{lemme-relatif}).
\end{proof}

\subsection{Connection with the Nash-Kuiper approach}
\label{Nash-Kuiper-historique}

\noi
The Nash-Kuiper $C^1$-Embedding Theorem was a prominent inspiration for the conception of Convex Integration Theory. The proof of this theorem relies on the resolution of a series of differential relations to construct a sequence of embeddings converging toward a $C^1$-isometric embedding. Each embedding differs from the previous one by a deformation localized on a chart $(U,\varphi=(\varphi^1,\dots,\varphi^m) )$ of $M$, $\dim\; M=m$. The differential problem to be solved is to build a $\epsilon$-isometric embedding $f_{\epsilon}:(U,\mu)\rightarrow \E^n$, $n>m$, where 
\begin{eqnarray*}
\mu := f^* \langle \cdot, \cdot \rangle+ \rho\, \ell \otimes \ell
\end{eqnarray*}
and $f$ is the previous map, $\rho:U\rightarrow\R_{>0}$ is a given map and $\ell=\sum_{i=1}^m c_id\varphi^i$, $c_i\in\R$, is a given linear form which is constant on the coordinate system (more details in section~\ref{sectionAppl2}). The differential relation under consideration is the following one
$$\is(\epsilon,\mu):=\{(x,y,L)\;|\; \|L^*\langle\cdot, \cdot \rangle-\mu_x\|<\epsilon\}$$
where $\|\cdot\|$ is the norm on bilinear forms induced by $\mu.$ In~\cite{Nash54} this differential relation is solved under the assumption that the codimension is greater than 2 by a spiraling perturbation of $f$ (the {\it step device}). In \cite{Kuiper55}, it is approximately solved in codimension one by a sinusoidal deformation (the {\it strain}). In~\cite{Lellis-Szeke,CdLdR}, this sinusoidal deformation is modified to obtain a true solution of $\is(\epsilon,\mu)$ (the {\it ansatz}). We show in the following how to solve $\is(\epsilon,\mu)$ by a single Corrugation Process in a direction $u$ dual to $\ell$. The resulting formula is similar to the ansatz and will be used in the proof of Theorem~\ref{th2} in section~\ref{sectionAppl2}.\\

\begin{figure}[!ht]
\centering
\includegraphics[scale=0.45]{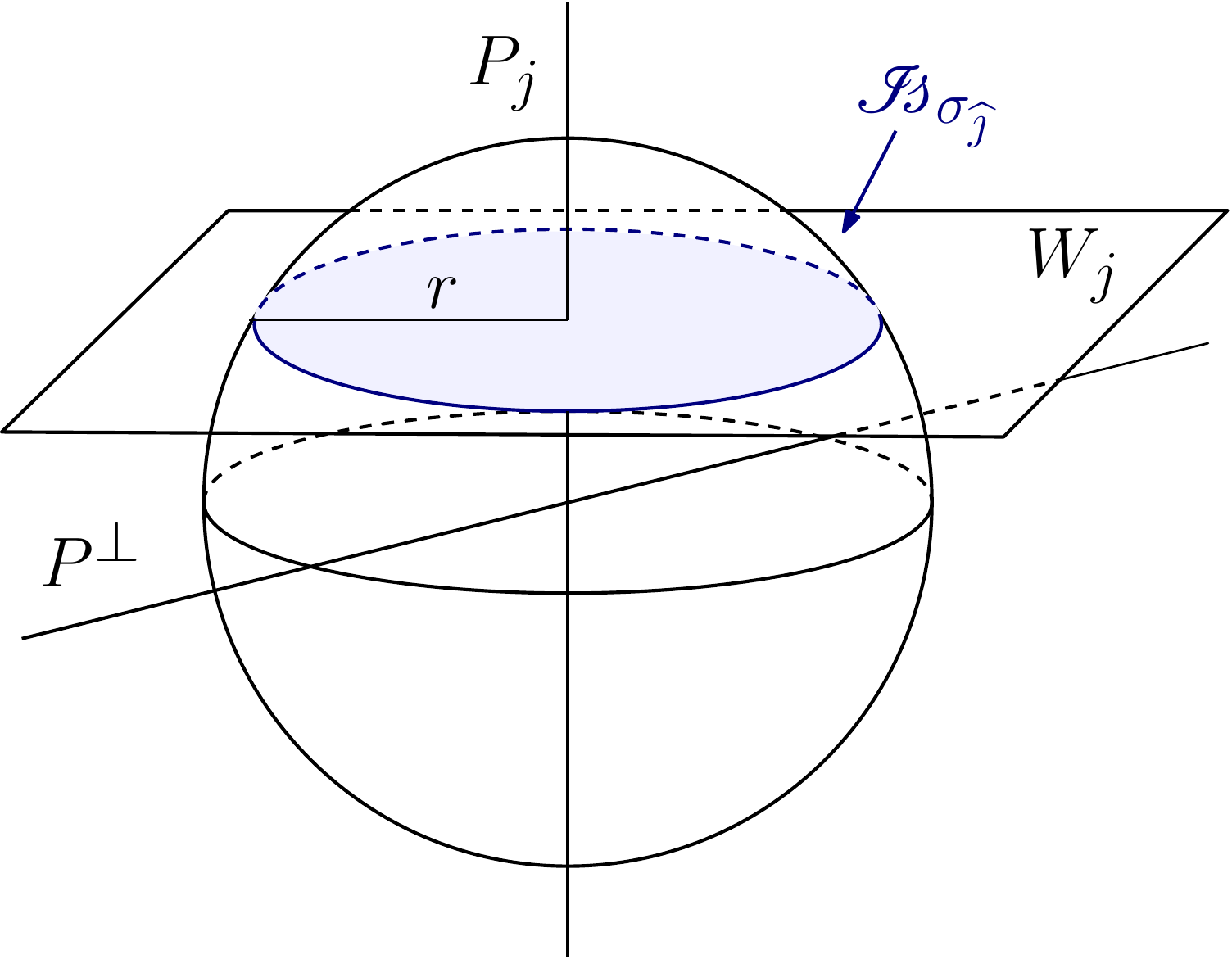}
\includegraphics[scale=0.7]{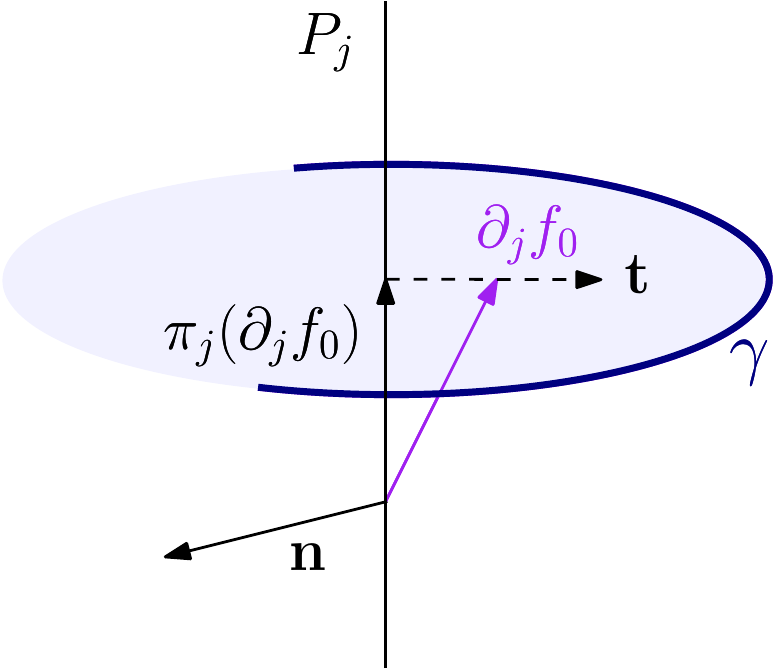}
\caption{\footnotesize{ \textbf{Left:} The slice of $\is(0,\mu)$ and its Convex hull, \textbf{Right:} the loop $\gamma$.}}
\label{image_relation_isometrie}
\end{figure}

\noi
For short we write $\rel$ instead of $\is(0,\mu)$. The space $\rel(\sigma,\ell,u)$ is described in \cite{Gromov-book} p. 202 and \cite{Spring-book} p. 194. It is the intersection of the $(n-1)$-dimensional sphere of radius $\sqrt{\mu(u, u)}$ with the affine $(n-m+1)$-plane 
\begin{eqnarray*}
\Pi:=\{ v \in \R^n \;|\;\langle v, L(u_0) \rangle = \mu_x (u,u_0) \; \forall u_0\in \ker \ell\}
\end{eqnarray*}
where as usual $\sigma=(x,y,L)$. It ensues that $\rel(\sigma,\ell,u)$ is a $(n-m)$-dimensional sphere of $\R^n$. Let $P:=L(\ker\ell)$. For every $v\in\R^n$ we write $v=[v]^P+[v]^{P^{\perp}}$ the decomposition of $v$ in $P\oplus P^{\perp}.$ \\ 

\noi
\textbf{The subsolution.--}  We consider sections $x\mapsto\s(x)=(x,f(x),L_x)$ of the form
$$L_x:=df_x+(v(x)- df_x(u))\otimes\ell$$ 
where $f$ is an embedding and $v:U\rightarrow \E^n$ is to be determined so that $\s\in\Gamma(\rel)$. By the very definition of $L$ we have $(L_x^*\langle\cdot, \cdot\rangle)(u_1,u_2)=\mu_x(u_1,u_2)$ for every couple $(u_1,u_2)\in \ker\ell\times \ker \ell$. If $u_1\in \ker \ell$, we also have $(L_x^*\langle\cdot, \cdot\rangle)(u_1,u)=\mu_x(u_1,u)$ if and only if $[v(x)]^P = [df_x(u)]^P$, and $(L_x^*\langle\cdot, \cdot\rangle)(u,u)=\mu_x(u,u)$ if and only if $\|v(x)\|^2=\mu_x(u,u).$ We choose
$$v(x):=[df_x(u)]^P+r(x){\bf t}(x)$$
with $\textbf{t} =\frac{[df(u)]^{P^{\perp}}}{\|[df(u)]^{P^{\perp}} \|}$ and $r= \sqrt{\mu(u,u) - \|[df(u)]^P\|^2}$. Observe that $P=L_x(\ker \ell)=df_x(\ker \ell)$ does not depend on $v$ and that $[df(u)]^{P^{\perp}}$ never vanishes since $f$ is an embedding. We also have
$\mu(u,u)=\|df(u)\|^2+\rho\geq \|[df(u)]^P\|^2$. Thus $v$ is well defined. For such a choice of $v$, it is readily checked that $\s(x)\in\rel$ and $df_x(u)\in \iconv\;\rel(\s(x),\ell,u)$ for every $x\in U.$ \\

\noi
\textbf{The loop family.--} We consider the following loop family
\begin{eqnarray}\label{loopIs}
\gamma(\cdot,t) &:=& r\cos(\alpha \cos(2\pi t))\textbf{t}
+ r\sin(\alpha \cos(2\pi t))\textbf{n} + [df(u)]^P 
\end{eqnarray}
where $\textbf{n}$ is any unit normal to $f$ over $U$. As $$\int_0^1\cos(\alpha \cos(2\pi t))dt=J_0(\alpha),\;\;\int_0^1\sin(\alpha \cos(2\pi t))dt=0$$
we choose 
$$\alpha := J_0^{-1} \left( \frac{\|[df(u)]^{P^{\perp}}\|}{r} \right)$$
to insure $\overline{\gamma}(x) =df_x(u_x)$.
Note that the argument of $J_0^{-1}$ lies between 0 and 1 because $$\|[df(u)]^{P^{\perp}}\|^2=\|df(u)\|^2-\|[df(u)]^P\|^2\leq \mu(u,u)-\|[df(u)]^P\|^2=r^2.$$

\noi
\textbf{Analytic expression of the Corrugation Process} Let $\pi=\sum_{i=1}^m c_id\varphi^i$ so that $d\pi=\ell.$ It ensues from our choice of $\gamma$ the following lemma.

\begin{prop}\label{kuiper_formula_is}
Let $f:(U,\mu)\rightarrow\E^n$ be an immersion and $\textbf{n}:U\rightarrow\E^n$ be a normal vector field to $f$. If $N$ is large enough, the Corrugation Process 
\begin{eqnarray*}
CP_{\hgamma}(f,\pi, N)= f + \frac{r}{N}K_c(\alpha,N\pi)\textbf{t} +\frac{r}{N}K_s(\alpha,N\pi) \textbf{n}
\end{eqnarray*}
defines an $\epsilon$-isometric immersion ($K_c$ and $K_s$ are given in the introduction).
\end{prop}

\section{Proof of Theorem~\ref{th1}}\label{proof_th1}
\subsection{Proof of point 1 (codimension 1 immersions)}\label{codim1}

We put $\im = \im(M,W)$, $m=\dim M$ and $n=\dim W =m+1$. We recall that the space $\im$ for codimension $1$ immersions is open, ample and given by 
$$\im:=\{(x,y,L) \; |\; \mbox{rank}\; L =m\}\subset J^1(M,W).$$
Let $\sigma=(x,y,L)\in \im$. For $\lambda\in T_x^*M$ and $u\in T_xM$ such that $\lambda(u)=1$, we define the slice
$$ \im(\sigma,\lambda,u) := Conn_{L(u)}\{ v\in T_yW\; |\; (x,y,L_v)\in\im\}$$
where $L_v:=L+(v-L(u))\otimes\lambda$. It is readily seen that $L_v(T_xM)=L(\ker\lambda) + \R v$ and thus $\im(\sigma,\lambda,u)$ is the complementary of $P:=L(\ker\lambda)$ in $T_yW$. Note that as $n>m$, the convex hull $\iconv \, \im(\sigma,\lambda,u)$ is the whole space $T_yW$, so each formal solution $\sigma$ is a subsolution. Let $\pi:U\subset M \rightarrow\R$ be a submersion and $u:U\rightarrow TM$ be such that $d\pi_x(u_x)=1$.\\

\noi
\textbf{Support plane of the loop family.--} Let $(\sigma,w)\in \iconv\;\im(\sigma,d\pi_x,u_x)$. Since $\sigma=(x,y,L)$ is in the relation $\im$, the rank of $L$ is maximal, then 
$$P := L(\ker d\pi_x)$$
is a $(m-1)$-dimensional vector subspace of $T_{y}W$, thus $P$ is of codimension $2$ and $L(T_x M)$ is of codimension $1$. Observe that $L(T_xM)^{\perp}$ inherits an orientation from the one of $L(T_xM)$ and $T_yW$. Let $\nu$ be the unique unit vector of $L(T_xM)^{\perp}$ inducing its orientation. Let $\Pi$ be the subspace spanned by $L(u)$ and $\nu$, we have $P\oplus\Pi = T_y W$. We will build a loop family $\bgamma$ in the affine space $w+\Pi$.\\

\noi
\textbf{Expression of the loop family.--} Let $(\sigma,w)\mapsto\mathcal{e}(\sigma,w)$ be the section of $q^*E \rightarrow \iconv(\im, d\pi,u)$ defined by
\begin{eqnarray*}
\begin{array}{llll}
\mathcal{e}(\sigma,w): & \R^3 & \longrightarrow & T_y W \\
& e_1 & \longmapsto & \textbf{r} L(u) /\|L(u)\|\\
& e_2 & \longmapsto & \textbf{r}\nu\\
& e_3 & \longmapsto & w
\end{array}
\end{eqnarray*}
and where ${\bf r} ={\bf r}(\sigma,w)$ is any $C^{k-1}$ map strictly greater than the distance between $w$ and its $P$-component $[w]^P = P\cap (w+\Pi)$, in particular  ${\bf r}$ never vanishes. 
This condition on ${\bf r}$ ensures that the circle $C_0$ of center $w$ and of radius ${\bf r}$ parametrized by 
$$t \mapsto \cos(2\pi t) {\bf e}_1 + \sin(2\pi t) {\bf e}_2 + {\bf e}_3$$
with ${\bf e}_i:=\mathcal{e}(\sigma,w)(e_i)$ does not intersect the vector space $P$ (see Figure~\ref{image_cercle_immersion}).
\begin{figure}[!ht]
\centering
\includegraphics[scale=0.85]{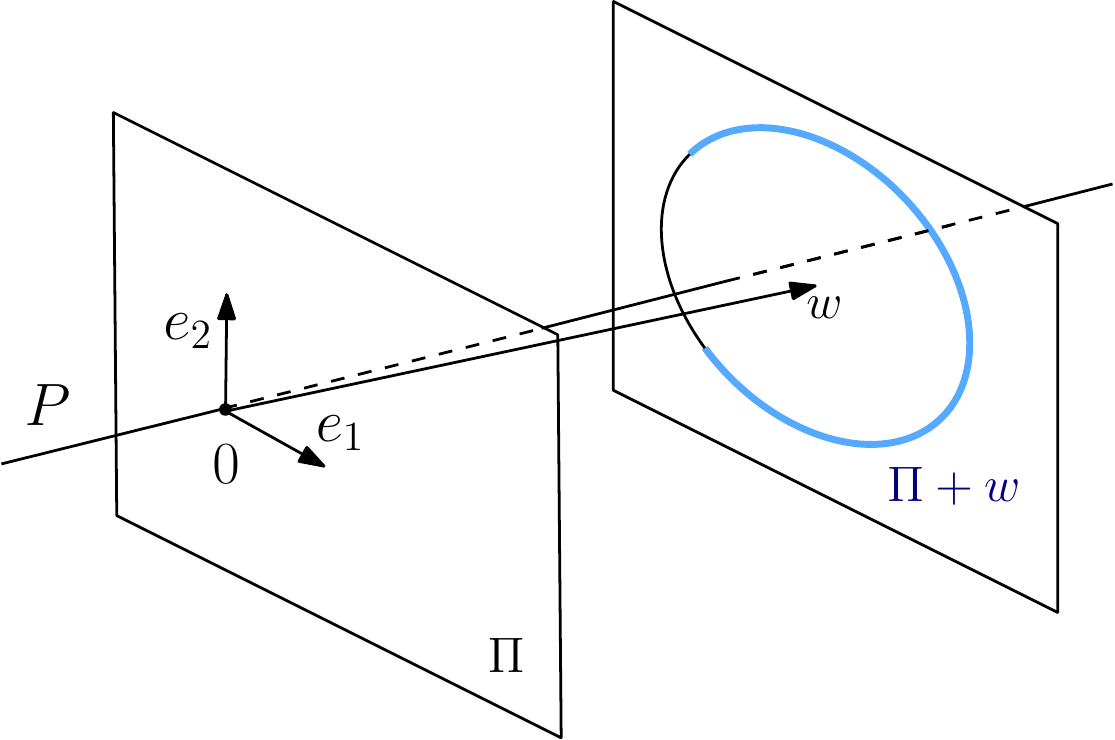}
\caption{\footnotesize{\textbf{Proof of Point 1 Theorem \protect \ref{th1}:} in blue the arc of circle of angle $\alpha_0$ on the circle $C_0$ of center $w$ and of radius $r$ inside the $2$-plane $\Pi+w$.}}\label{image_cercle_immersion}
\end{figure}
We consider the loop pattern given by
\begin{eqnarray*}
\begin{array}{lclc}
c: & [0,\alpha_0] \times\R/\Z & \longrightarrow & \R^3\\
& (\alpha,t) & \longmapsto & \Big(\cos(\alpha \cos 2\pi t)- J_0(\alpha), \sin(\alpha \cos 2\pi t),1\Big)
\end{array}
\end{eqnarray*}
and the loop family 
\begin{eqnarray}\label{bgamma_imm}
\bgamma(\sigma,w)(t):= (\cos \left(\alpha \cos 2\pi t \right)-J_0(\alpha))\, {\bf e}_1 + \sin \left(\alpha \cos 2\pi t \right) {\bf e}_2+{\bf e}_3.
\end{eqnarray}
Note that if $\alpha = \alpha_0$, the image of $\bgamma$ is an arc of circle of angle $\alpha_0$ whose image lies inside the circle $C_0$ of center ${\bf e}_3=w$. Moreover if $\alpha=0$, we have $c(0,t) = (0,0,1)$, then $\bgamma(\sigma,w) \equiv w$.\\

\FloatBarrier

\noi
\textbf{Relative property.--} Let $\delta > 0$. The image of $\bgamma$ is an arc of circle of angle $2\alpha$ and of center $\Omega={\bf e}_3-J_0(\alpha){\bf e}_1$. Our goal is to define a function $\alpha:\iconv(\im,d\pi_x,u_x)\rightarrow [0,\alpha_0]$ such that $\bgamma$ is $\delta$-relative. To do so we introduce three subspaces of $\iconv(\im,d\pi_x,u_x)$:
\begin{eqnarray*}
\begin{array}{lll}
Z_{\Delta} &=&\{(\sigma,w) \;|\; \delta\leq d(w) \mbox{ and } w= L(u)\}\vspace*{1mm}\\
Z_1(\epsilon) &=& \{(\sigma,w)\;|\; \frac{\delta}{2} < d(w) \mbox{ and } dist(L(u),w) < \epsilon \}\\
Z_0 &=& Z_1(\epsilon)^C.
\end{array}
\end{eqnarray*}
where $d(w):=dist(w, \im(\sigma,d\pi_x,u_x)^C)=\dist(w,P)$ and $\epsilon>0$ will be chosen latter. We consider any smooth map $\alpha
: \iconv(\im,d\pi_x,u_x) \rightarrow [0,\alpha_0]$ such that
$$\left\{
\begin{array}{ll}
\alpha(\sigma,w)=0 & \mbox{ if }(\sigma,w)\in Z_{\Delta} \\
\alpha(\sigma,w)=\alpha_0 & \mbox{ if } (\sigma,w)\in Z_0.
\end{array}
\right.$$
By construction of $\alpha$, it is immediate that the loop family $\bgamma$ is $\delta$-relative. Indeed, $t\mapsto \bgamma(\sigma,w)(t)$ is constant equal to $w$ if $dist(w, P)\geq \delta$ and $w= L(u)$. Remark also that $\bgamma(\sigma,w)$ is an arc of circle of angle $2\alpha_0$ if $(\sigma,w)\in Z_0$. \\

\begin{figure}[!ht]
\centering
\includegraphics[scale=0.5]{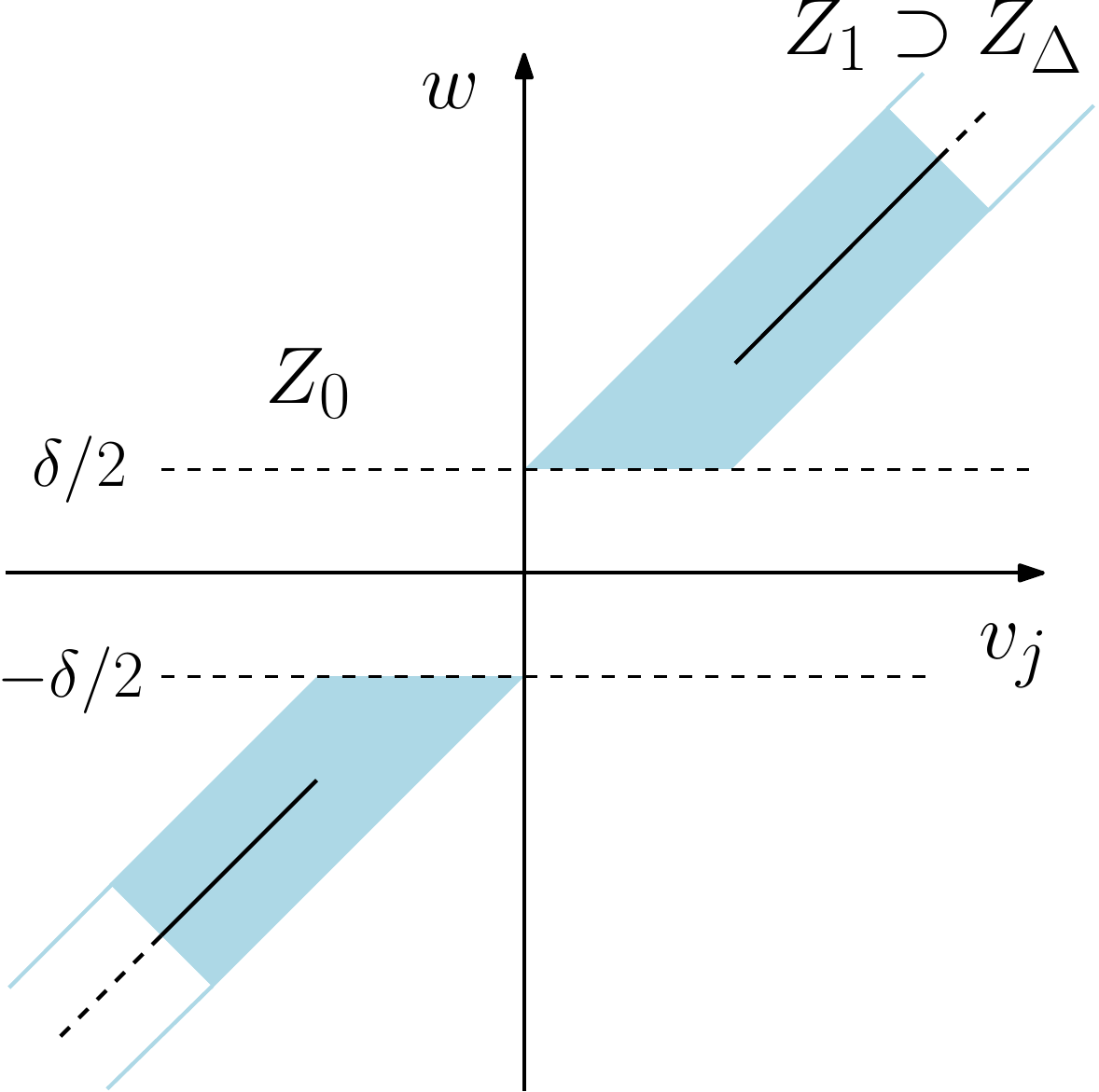}
\caption{\footnotesize{ \textbf{Values of $\alpha$ in $\iconv(\im,d\pi_x,u_x)$}: $\alpha=\alpha_0$ over $Z_0$ and $\alpha=0$ over $Z_{\Delta}$. In between, $\alpha$ is a smooth interpolation.}}\label{image_domaine_interpolation_immersion}
\end{figure}

\noi
\textbf{Proof of the surrounding property.--} To prove that the loop family built with the previous pattern $c$ and the previous map $\mathcal{e}$ is surrounding, we have to prove  Points (1), (2) and (3) of Definition~\ref{def_surrounding}.\\

\noi
\textbf{Proof of Point (1).--} We first check Point $(1)$ of Definition~\ref{def_surrounding} i.e. that $t\mapsto \bgamma(\sigma,w)(t)$ is a loop in $\im(\sigma, d\pi_x, u_x)$. Observe that if $(\sigma,w)\in Z_{\Delta}$ then $\bgamma(\sigma,w)$ is a constant map whose value $w$ is in $\im(\sigma, d\pi_x, u_x)$ by definition of $Z_{\Delta}$. Observe also that over $Z_0$, the image of $\bgamma(\sigma,w)$ lies on the circle $C_0$ of radius ${\bf r}$ and center $w$. The choice of $r$ implies that this circle is included in $\im(\sigma, d\pi_x, u_x)$. It remains to show that if $\epsilon$ is small enough and if $(\sigma,w)\in Z_1(\epsilon)$, the image of $\bgamma(\sigma,w)$ lies inside $\im(\sigma, d\pi_x, u_x).$  We first consider the punctured diagonal $\Delta^*=\{ (L(u),L(u)), L(u)\neq 0\}$ (note that $Z_{\Delta}\subset \Delta^*$). Over $\Delta^*$, we have 
\begin{eqnarray*}
\bgamma(\sigma,w)(t) = \left(1+ \Big( \cos(\alpha\cos 2\pi t)-J_0(\alpha) \Big) \displaystyle \frac{\textbf{r}}{\|L(u)\|}\right) L(u)+\sin(\alpha\cos 2\pi t){\bf r}\nu
\end{eqnarray*}
Observe that $\bgamma(\sigma,w)(t)\not\in P\iff\bgamma(\sigma,w)(t)\neq 0$ and it is readily checked that $\bgamma(\sigma,w)(t)$ never vanishes. By continuity, this is still true on the neighborhood of $Z_{\Delta}$ and thus on $Z_1(\epsilon)$ for $\epsilon>0$ small enough (see Figure~\ref{image_domaine_interpolation_immersion}). This finishes the proof of Point $(1)$.\\

\noi
\textbf{Proof of Point (2).--} By the definition of the Bessel function, $\overline{\bgamma}(\sigma,w) = {\bf e}_3=w$ and therefore $\bgamma$ satisfies the Average Constraint, i.e. Point (2).\\

\noi
\textbf{Proof of Point (3).--} To prove this point, we describe a canonical homotopy between the base point
\begin{eqnarray*}
\bgamma(\sigma, w)(0) = (\cos(\alpha) - J_0(\alpha) )\textbf{e}_1 + \sin(\alpha)\textbf{e}_2 +w
\end{eqnarray*}
and $L(u)$. To do so, we distinguish two cases whether $(\sigma,w)$ lies inside $Z_0$ or $Z_1(\epsilon)$.\\

\noi
\textit{Case $(\sigma,w)\in Z_{0}$.}  In that case $\alpha(L(u),w) = \alpha_0$. We first consider a homotopy $h_{1/4}$ between $\bgamma(\sigma,w)(0)$ and $\bgamma(\sigma,w)(1/4)$ given on $\bgamma$ by a homotopy between $0$ and $1/4$. Remark that $\bgamma(\sigma,w)(1/4) = {\bf e}_1 + w$. Then we can define the obvious affine homotopy from $\bgamma(\sigma,w)(1/4)$ to $L(u)$ by setting
\begin{eqnarray*}
h(t) := t \,({\bf e}_1 + w) + (1-t)L(u).
\end{eqnarray*}
We decompose $w= [w]^P + w_1\textbf{e}_1 + e^{\perp}$ with $e^{\perp}\in (P \oplus Span(L(u)))^{\perp}$. Because $\textbf{r}> \dist([w]^P,w),$ we have $\|w_1 {\bf e}_1\|< \textbf{r}$ and thus $|w_1|<1$. We now can write
$$h(t)=(t+tw_1+(1-t)\textbf{r}^{-1}\|L(u)\|)\, {\bf e}_1+ t e^{\perp}.$$
Since $t+tw_1>0$, the ${\bf e}_1$-component of $h(t)$ never vanishes so $h(t)$ is a canonical homotopy of $\im(\sigma, d\pi_x, u_x)$ joining the base point $\bgamma(\sigma,w)(1/4)$ to $L(u)$. It is enough to consider the homotopy $h\circ h_{1/4}$ to conclude.\\

\noi
\textit{Case $(\sigma,w)\in Z_1(\epsilon)$.} The previously defined homotopy $h_{1/4}$ allows to join $\bgamma(\sigma,w,\alpha)(0)$ to $\bgamma(\sigma, w, \alpha)(1/4)$ while staying in $\im(\sigma, d\pi_x, u_x)$. But over $Z_1(\epsilon)$, the function $J_0(\alpha)$ varies from $1$ to $0$ and thus $\bgamma(1/4) = (1-J_0(\alpha))\textbf{e}_1 +w$ is not $\textbf{e}_1 +w$ in general. We thus introduce an extra homotopy $h_{\alpha_0}$ defined by
\begin{eqnarray*}
h_{\alpha_0}(t) := \bgamma(\sigma, w, (1-t)\alpha + t \alpha_0)(1/4)
\end{eqnarray*}
to join $\bgamma(\sigma, w, \alpha)(1/4)$ to $\bgamma(\sigma, w, \alpha_0)(1/4).$ Note that our choice of $\epsilon$ in the definition of $Z_1(\epsilon)$ ensures that the image of $h_{\alpha_0}$ lies inside $\im(\sigma, d\pi_x,u_x)$. The homotopy $h \circ h_{\alpha_0} \circ h_{1/4}$ joins $\bgamma(\sigma, w)(0)$ to $L(u)$. This concludes Point (3) of Definition \ref{def_surrounding}.\\

\noi
It ensues that $\bgamma$ is a relative surrounding loop family which is $c$-shaped. This proves Point 1 of Theorem~\ref{th1}.

\subsection{Proof of point 2 (totally real immersions)}
\label{proof_point2}

\noi
Let $m=\dim(M)$ and $\im_{TR}\subset J^1(M,W)$ be the differential relation of totally real immersions
$$\im_{TR} := \{(x,y,L) \; |\;  L(T_xM)\oplus JL(T_xM)=T_yW \}$$
and let $\sigma=(x,y,L)\in \im_{TR}$. We consider the slice
$$ \im_{TR}(\sigma,\lambda,u) = Conn_{L(u)}\{ v\in T_yW\; |\; (x,y,L_v)\in\im_{TR}\}$$
where $L_v:=L-(v-L(u))\otimes\lambda$, $\lambda\in T_x^*M$ and $u\in T_xM$ such that $\lambda(u)=1.$ Let $P:=L(\ker \lambda)+JL(\ker \lambda)$. Observe that $L_v(\ker\lambda)=L(\ker\lambda)\subset P$, $L_v(u) =v$ and $JP=P$. If $v\in P$ then $ L_v(T_x M) = L_v (\ker \lambda) +\R v \subset P$, so
\begin{eqnarray*}
L_v(T_xM) + JL_v(T_xM) \subset P \varsubsetneq T_y W
\end{eqnarray*}
and $L_v$ is not a totally real linear map. If $v\not\in P$, we have $Jv\not\in P$, then $L_v$ is totally real. We conclude that $v\in \im_{TR}(\sigma,\lambda,u)$ if and only if $v\notin P$. We have $\dim\;P=2m-2$ because $L:T_xM\rightarrow T_yW$ is a totally real linear map. It ensues that the space $\im_{TR}(\sigma,\lambda,u) $ is the complementary of the codimension 2 linear subspace $P$ in $T_yW$.\\

\noi
The slice $\im_{TR}(\sigma,\lambda,u)$ is thus completely analoguous to the slices of the relation of codimension one immersions. As a consequence, the construction of $\bgamma$ done for codimension one immersions fully applies for totally real immersions (we replace the unit normal vector $\nu$ by ${\bf r}JL(u)/\|JL(u)\|_h$ where $h$ is an auxiliary Hermitian metric). This shows that $\mathcal{I}_{TR}$ is a Kuiper relation with respect to~$c$ and that $\bCP(f,\pi,N)$ is totally real if $N$ is large enough.

\subsection{Expression of the Corrugation Process formula}\label{def_KcKs}

According to Proposition~\ref{prop_noInt}, the Corrugation Process applied on any subsolution $\s$ generates a map
\begin{eqnarray*}
\CP(f_0,\pi,N)(x)
&=&  \exp_{f_0(x)}\left(\frac{1}{N} C(\alpha(x),N\pi(x))\cdot e(x)\right)
\end{eqnarray*}
where $f_0=bs\;\s:M\rightarrow W$, $e(x):={\bf e}(\s(x),d f_0(u_x))$ and 
\begin{eqnarray*}
C(\alpha, t) &=& \int_{u=0}^t \Big( c(\alpha,u)-(0,0,1) \Big) du\\
 &=& \left(\int_{u=0}^t \Big( \cos(\alpha \cos 2\pi u)-J_0(\alpha)\Big )du, \int_{u=0}^{t} \sin(\alpha \cos 2\pi u ) du  ,0 \right)\\
 &&\\
 &=& (K_c(\alpha,t), K_s(\alpha,t),0).
\end{eqnarray*}

\noi
Observe that $K_c$ is $\frac{1}{2}$-periodic in $t$ and that $K_s(\alpha,t+\frac{p}{2})=(-1)^pK_s(\alpha,t)$ for $p\in\N$.

\section{Application 1: Immersion of $\R P^2$}\label{sectionApplImmersion}

\subsection{Desingularization of the Plücker's conoid}\label{corrugated_PC}

\noi
{\bf The initial map.--} We consider the following parametrization $f_0:D=[-3,3]\times[0,1]\rightarrow\R^3$ of the Plücker's conoid
\begin{eqnarray*}
f_0(x_1,x_2) = \left(x_1\cos(\pi x_2),\ x_1\sin(\pi x_2),\ \frac{1}{2}\cos(2\pi x_2)\right).
\end{eqnarray*}

\begin{figure}[!ht]
\centering
\includegraphics[scale=0.15]{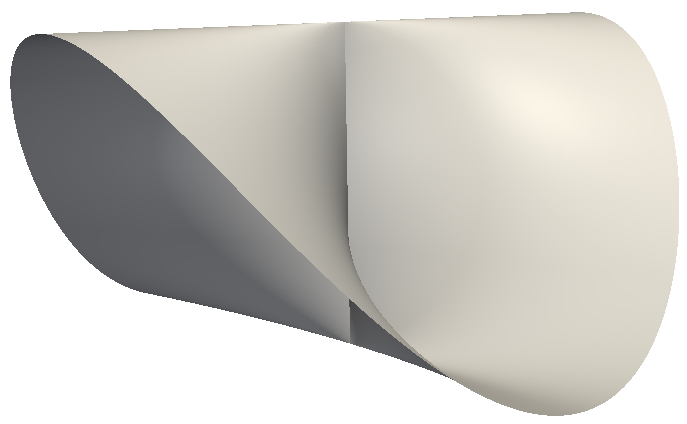}$\hspace{0.06\linewidth}$
\includegraphics[scale=0.2]{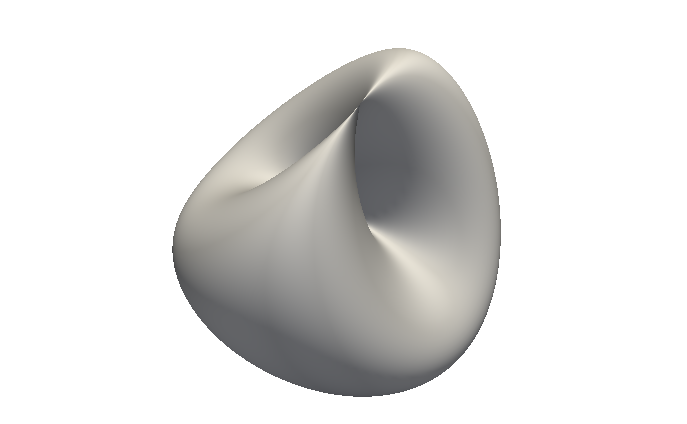}
\caption{\footnotesize{\textbf{Plücker's conoid and Cross-cap}}}\label{image_conoid}
\end{figure}

\noi
\textbf{The direction.--} The map $f_0$ has two singular points $x = (0,0)$ and $x = (0,\frac{1}{2})$ on which $\partial_{2} f_0$ is zero. The removal of these singularities will be performed by a relative Corrugation Process in the direction $u=\partial_2$. Let
\begin{eqnarray*}
K := \Big( [-3,-2] \cup \  [2,3]\Big)  \times [0,1]\,\mbox{ and }\;
\mathfrak{Op}(K) := \Big( [-3,-1[ \cup \  ]1,3]\Big)  \times [0,1].
\end{eqnarray*}
Precisely, we shall built a section $\s :D\rightarrow \im$ coinciding with $j^1f_0$ on $\mathfrak{Op}(K).$\\

\noi
\textbf{The subsolution.--}  For each $x_2\in [0,1]$, we consider the rotation
$R_{x_2,\theta}$ of angle $\theta$ in the oriented plane
$$\Pi_2(-1,x_2):=Span( \partial_{2}f_0(-1,x_2),(\partial_{2}f_0\wedge\partial_{1} f_0)(-1,x_2)).$$
Let $\theta$ be a smooth map such that $x_1\mapsto \theta(x_1,x_2)$ is an interpolation between $0$ and $\theta_{max}(x_2)$ where
$$\theta_{max}(x_2):=\mathrm{angle}(\partial_{2}f_0(-1,x_2),\partial_{2}f_0(1,x_2)).$$
We write a section $\s=(x,y,L)$ under the form $x\mapsto (x,y(x),v_1(x),v_2(x))$ where $v_1=L(\partial_1)$ and $v_2=L(\partial_2)$. We define $v_1$ and $v_2$ to be  
\begin{eqnarray*}
 v_1(x_1,x_2) & := & \partial_{1}f_0(x_1,x_2),\\
 v_2(x_1,x_2) & := & R_{x_2,\theta(x_1,x_2)}(\partial_{2} f_0(-1,x_2))
\end{eqnarray*}
on $D\backslash \mathfrak{Op}(K)$ and to be $v_i=\partial_i f_0$ elsewhere. Since $\partial_{2}f_0(1,x_2)$ is in the plane $\Pi_2(-1,x_2)$, the map $v_2$ is continuous on $D.$ It implies that $x\mapsto \s (x)=(x,f_0(x),\partial_{1}f_0(x),v_2(x))$ is a continuous section of $\mathcal{I}$ which coincide with $j^1f_0$ over $\mathfrak{Op}(K).$\\

\noi
\textbf{The relative loop family.--} We set
$$e_1(x) := r\frac{v_2(x)}{ \| v_2(x)\|},\;\;\;e_2(x): = r\frac{ v_2(x) \wedge v_1(x) }{\| v_2(x) \wedge v_1 (x) \|}$$
with
\begin{eqnarray*}
r := \sup_{x\in D\backslash \mathfrak{Op}(K)} \| \partial_{2} f_0(x) \| +\frac{1}{2} = \sqrt{2}\pi +\frac{1}{2}.
\end{eqnarray*}

\noi
Let $\alpha$ be any interpolating smooth function such that $\alpha\equiv 0$ on $K$ and $\alpha\equiv\alpha_0$ on $D\setminus\mathfrak{Op}(K)$. From the relative loop family $\bgamma$ of Section~\ref{codim1}, we derive the following expression for $\gamma$:
\begin{eqnarray*}
\gamma(x,t) :=& \Big(\cos(\alpha(x) \cos 2\pi t) - J_0(\alpha(x))\Big)e_1(x)\\
&+ \sin( \alpha(x) \cos 2\pi t) e_2(x) + \partial_{2}f_0(x). 
\end{eqnarray*} 

\noi
\textbf{The quotient condition.--} We consider the action of $G=\Z$ on $[-3,3]\times \R$ given by $k\cdot  (x_1,x_2) = ((-1)^kx_1,x_2+k)$. A fundamental domain for this action is $[-3,3]\times [0,1]$ and its quotient $\M^2=\R^2/G$ is a   Mobius strip. Observe that the parametrization $f_0$ is $G$-invariant and thus descends to the quotient. 
It is readily seen that $f_1=\CP(f_0,\partial_2,N)$ descends to the quotient if
\begin{eqnarray}\label{eq_MobiusCondition}
\Gamma(1\cdot x,Nx_2+N) = \Gamma(x,Nx_2).
\end{eqnarray}
By definition of $v_1$ and $v_2$, we have
\begin{eqnarray*}
e_1(k\cdot x) = e_1(x), \,\,\, e_2(k\cdot x) = (-1)^k e_2(x).
\end{eqnarray*}
Thus, if we choose $\alpha$ to be $G$-invariant, we obtain
\begin{eqnarray*}
\begin{array}{lll}
\Gamma\left(1\cdot x,N(x_2+1)\right)
&=& K_c\left(\alpha(x),N(x_2+1)\right)e_1(1\cdot x)+K_s\left(\alpha(x),N(x_2+1)\right)e_2(1\cdot x)\vspace*{1mm}\\
&=& K_c\left(\alpha(x),N(x_2+1)\right)e_1(x) - K_s\left(\alpha(x),N(x_2+1)\right)e_2(x)
\end{array}
\end{eqnarray*}
Since $K_c(\alpha,t+\frac{p}{2})=K_c(\alpha,t)$ and $K_s(\alpha,t+\frac{p}{2})=(-1)^p K_s(\alpha,t)$, it is enough to choose $N\in\N +1/2$ to fulfill condition~\ref{eq_MobiusCondition}

\begin{figure}[!ht]
\centering
\includegraphics[scale=0.14]{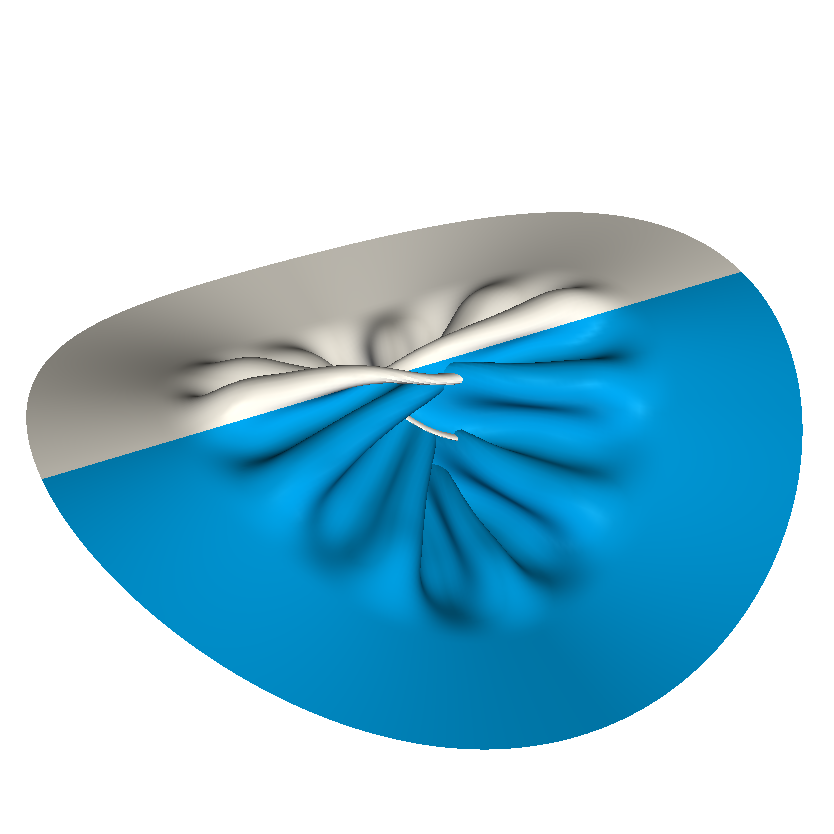}
\includegraphics[scale=0.14]{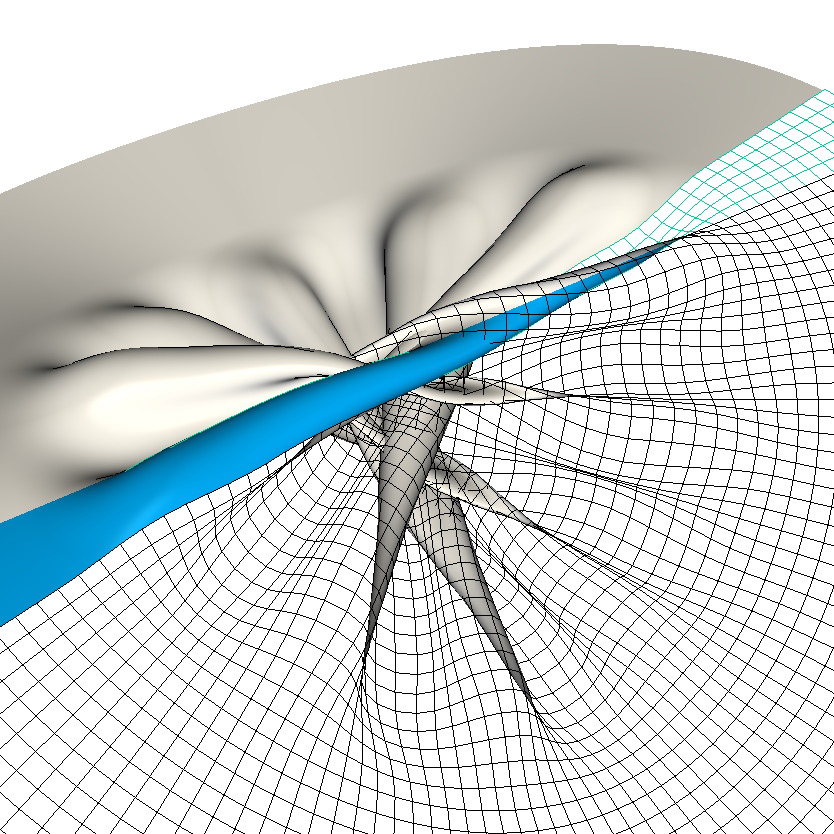}
\includegraphics[scale=0.14]{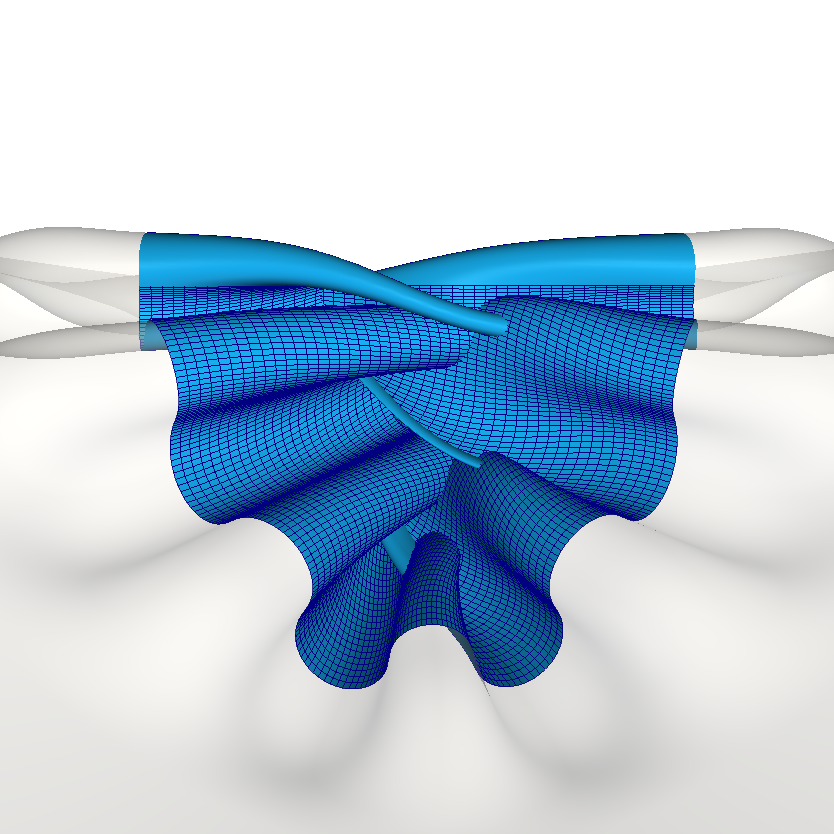}
\end{figure}
\begin{figure}[!ht]
\centering
\includegraphics[scale=0.14]{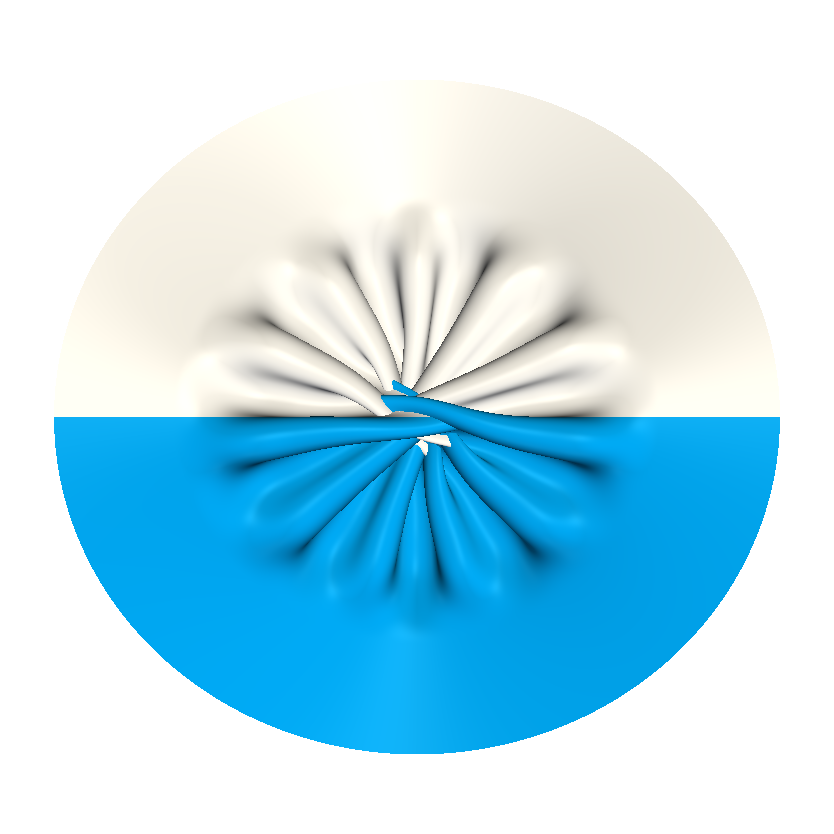}
\includegraphics[scale=0.14]{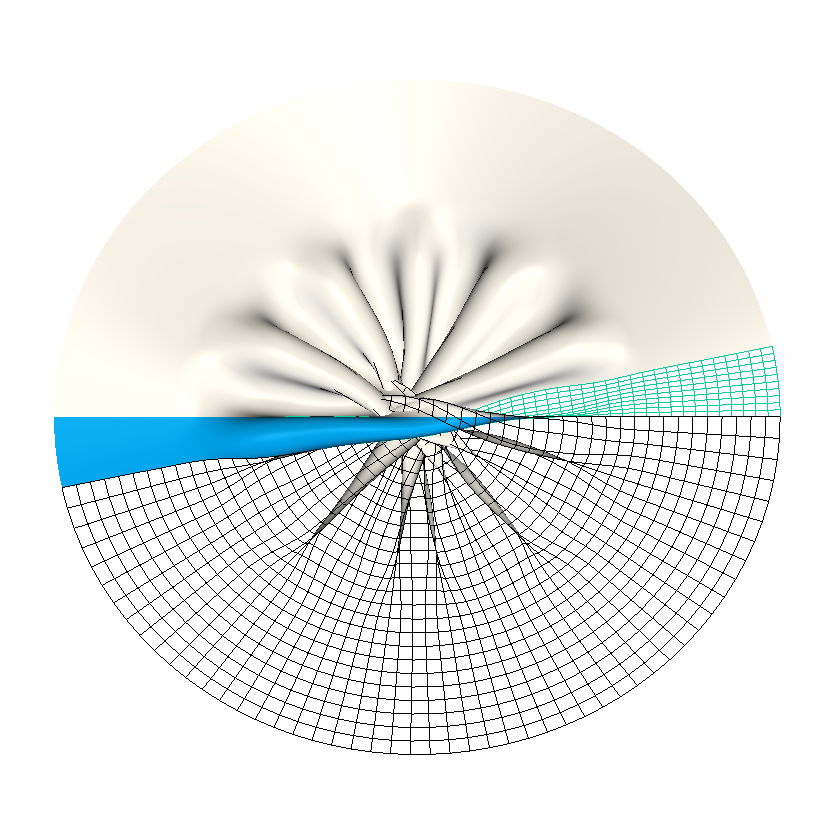}
\includegraphics[scale=0.14]{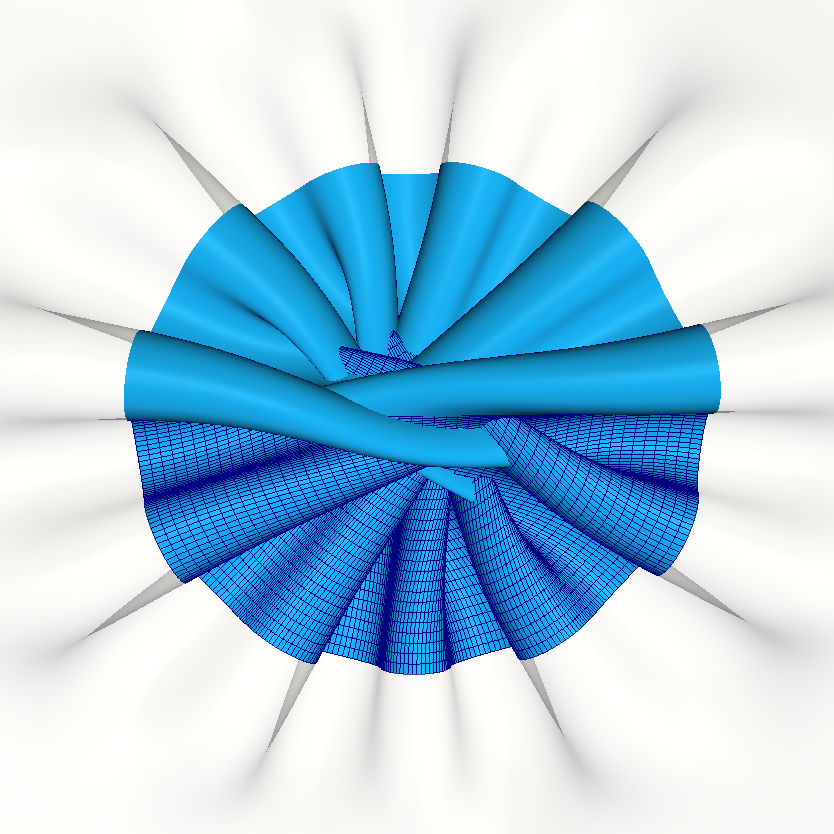}
\caption{ \footnotesize{\textbf{Desingularization of the Plucker Conoid} obtained as an image of $\M^2$ by $f_1$ with $N=5.5$ and $\theta(x_1,x_2) =  0.5\ ( \sin( 0.5\pi x_1 )+1 )\theta_{max}(x_2),$ $\alpha(x) = \frac{\alpha_0}{2}( \cos(\pi x_1 +\pi) + 1 )$.}}\label{image_conoid_local}
\end{figure}

\FloatBarrier

\noi
\textbf{The desingularized map.-- } This map is given by
\begin{eqnarray}\label{equ_conoid_desing}
f_1(x)=f_0(x)+\frac{1}{N}K_c(\alpha(x), Nx_2)e_1(x)+\frac{1}{N}K_s(\alpha(x), Nx_2)e_2(x)
\end{eqnarray}
for all $x\in D.$ The maps $K_c$ and $K_s$ are the functions defined in Subsection~\ref{def_KcKs}. According to Proposition~\ref{prop_noInt} the map $f_1$ is an immersion if $N$ is large enough.

\begin{figure}[!ht]
\centering
\includegraphics[scale=0.14]{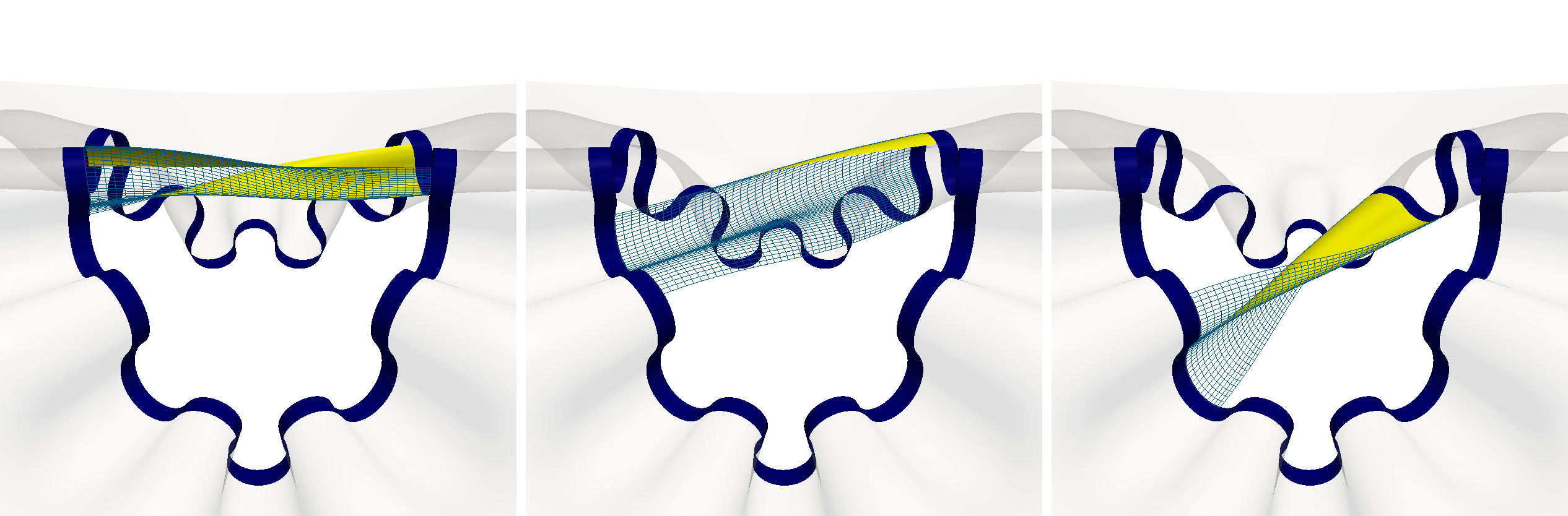}
\end{figure}
\begin{figure}[!ht]
\centering
\includegraphics[scale=0.14]{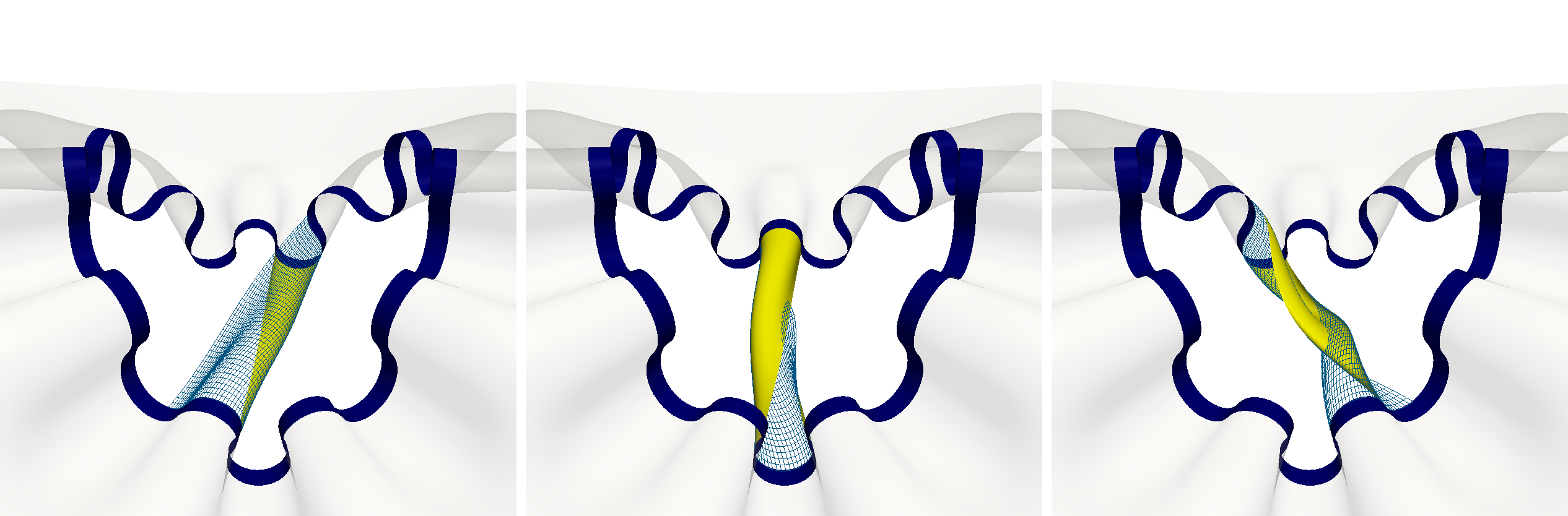}
\end{figure}
\begin{figure}[!ht]
\centering
\includegraphics[scale=0.14]{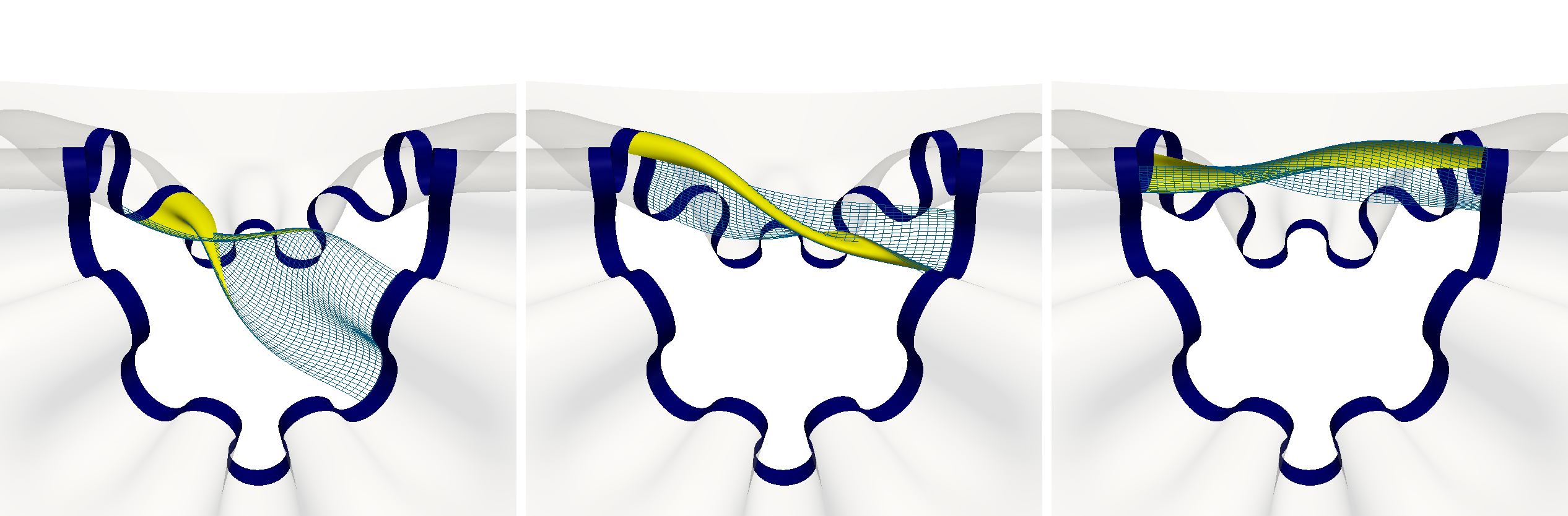}
\caption{ \footnotesize{\textbf{The spinning of corrugation:} note the similarity with the "tobacco pouch" surfaces of \protect \cite{Francis}, p.115.}}\label{image_conoid_centre}
\end{figure}

\FloatBarrier

\subsection{New Immersions of $\R P^2$}\label{RP2_close}

\noi
{\bf Immersions of $\R P^2$ via an inversion of $f_1$.--} Let us consider the entire Plücker conoid, that is the image of $f_0$ on the set $\R\times [0,1]$ (where $f_0$ is well-defined). It is well known that a representation of the projective space can be obtained by applying an inversion to the entire Plücker conoid. The resulting surface is bounded, its closure still has two pinch points and is called a cross-cap. We consider $\widetilde{f_1}:= f_1$ on $D$ and $\widetilde{f_1}= f_0$ on $(\R\times [0,1]) \backslash D$. To obtain an immersion of the real projective space it is enough to apply an inversion to $\widetilde{f_1}$ and to take its closure.\\

\noi
{\bf Immersions of $\R P^2$ via an extension of $f_1$.-- } Here, for purely aesthetic reasons, we avoid the use of an inversion and we choose to extend the corrugated Plücker conoid in such a way it closes up on itself like an hemisphere. Specifically, we parametrize the sphere of radius 2.5 by a map $S$ defined on $[-5,5]\times [0,1]$ and given by 
$$(x_1,x_2)\mapsto 2.5\left(\cos \pi x_2\sin \frac{\pi x_1}{5},\sin \pi x_2\sin \frac{\pi x_1}{5}, \cos\frac{\pi x_1}{5} \right).$$
We then define a map $F_1$ on the same domain by setting:
\begin{eqnarray}\label{equ_RP2_desing}
\left\{
\begin{array}{lll}
F_1(x_1,x_2)_{XY} &=& S(x_1,x_2)_{XY} +f_1(x_1,x_2)_{XY}-f_0(x_1,x_2)_{XY}\\
F_1(x_1,x_2)_Z &=& S(x_1,x_2)_Z + \beta(x_1)(f_1(x_1,x_2)_Z - 1)
\end{array}
\right.
\end{eqnarray}
if $(x_1,x_2)\in [-2.5,2.5]\times  [0,1]$ and $F_1=S$ otherwise. In this formula, $f_{XY}$ means the $X,Y$ components of $f$, $f_Z$ means the $Z$ components and the map $\beta$ is a smooth interpolation between $\beta(0)=1$ and $\beta(\pm 2.5)=0$ (see Figure~\ref{image_RP2}). 

\FloatBarrier

\section{Application 2: Totally real isometric embeddings}\label{sectionAppl2}

\subsection{Proof of Theorem~\ref{th2}}\label{subsection:proofthm2}
The proof is divided into two parts. The first one relies on the arguments of Nash-Kuiper \cite{Nash54, Kuiper55, EliaMisha-book} to construct a sequence of maps $(f_k)$ converging toward an $C^1$ isometric map $f_{\infty}: (M,g)\rightarrow (W,J,h)$. In the second part, we replace the Nash-Kuiper processes by a Corrugation Process to construct iteratively the sequence $(f_k)$. We use Corrugation Process $CP_{\bgamma}(f,\pi,N)$ defined in section \ref{proof_point2} to solve both the totally real relation $\im_{TR}$ and the isometric relation $\is(\epsilon,\mu)$ with $\mu= f^*h+\rho\, \d\pi\otimes d\pi$ (see Section~\ref{Nash-Kuiper-historique}). We use the fact that $\im_{TR}$ is Kuiper with respect to $c$ to control the geometry of each $f_k$ and to show that the limit map $f_{\infty}$ is totally real. \\

\noi
\textbf{First part:}  Let $\Delta:=g-f_0^*h$ be the isometric default, $(\delta_k)_{k\in\N^*}$ be an increasing sequence of positive numbers converging to 1 and
$$g_k=f_0^*h+\delta_k\Delta$$
be an increasing sequence of metrics converging toward $g$. The Nash-Kuiper proof involves in building an infinite sequence $(f_k)_{k\in\N}$ of maps such that each $f_k$ is approximatively isometric for $g_k$ and short for $g_{k+1}$:
$$f_k^*h\approx g_k\;\;\;\mbox{ and }\;\;\; f_k^*h\leq g_{k+1}.$$
This sequence is obtained by a succession of local deformations performed to reduce the isometric default in one direction (approximatively). This generates a finite number of recursively defined intermediary maps 
$$f_k=f_{k,0}, f_{k,1},\dots, f_{k,I(k)}=f_{k+1}$$
where $I(k)<+\infty$ depends on $\dim M=m$ and on the number of charts of a finite atlas $\{(U_a,\varphi_a)\;|\;a\in A\}$ of $M$. More precisely, let $(\psi_a)_{a\in A}$ be a partition of unity subordinate to $(U_a)_{a\in A}$. A first portion of the intermediary maps achieves the approximation of the increase $\psi_1(g_{k+1}-f_{k}^*h)$ over $U_1$, a second portion achieves the increase $\psi_2(g_{k+1}-f_{k}^*h)$ over $U_2$, etc. To do so, on each $U_a$ the desired increase is decomposed as a finite combination of squares of constant linear forms:
$$\psi_a(g_{k+1}-f_{k}^*h)=\sum_{j=1}^{j_{max}(a)}\rho_{a,j}\ell_{a,j}\otimes\ell_{a,j}$$
with positive coefficients $\rho_{a,j}:U_a\rightarrow \R_{\geq 0}$. The map $f_{k,1}$ is built to satisfy 
$f_{k,1}^*h\approx f_k^*h+\rho_{1,1}\ell_1\otimes\ell_1$, the map $f_{k,2}$ to satisfy $f_{k,2}^*h\approx f_{k,1}^*h+\rho_{1,2}\ell_2\otimes\ell_2$, etc. In particular, for $I=\sum_{a\in A}j_{max}(a)$ the map $f_{k,I}$ satisfies $$f_{k,I}^*h\approx f_k^*h+(g_{k+1}-f_{k}^*h)=g_{k+1}.$$
{\it In fine}, the fundamental step in this approach is the following: given a map $f:U\rightarrow W$, a positive coefficient $\rho:U\rightarrow \R_{\geq 0}$, a chart $(U,\varphi=(\varphi^1,\dots,\varphi^m))$ and a constant linear form $\ell=\sum_{i=1}^m c_id\varphi^i$ to construct an $\epsilon$-isometric map $f_{\epsilon}:(U,\mu) \rightarrow (W,h)$ where $\mu:=f^*h+\rho\ell\otimes\ell$. Eventually, the $C^1$ convergence of the sequence $(f_k)_{k\in \N}$ toward a $C^1$ isometric map $f_{\infty}:(M,g)\rightarrow (W,h)$ is insured by choosing a sequence of spiraling/oscillation numbers $(N_{k,i})$ that increases rapidly and a sequence $(\delta_k)_{k\in\N^*}$ such that $\sum \sqrt{\delta_k-\delta_{k-1}}<+\infty$.\\

\noi
\textbf{Second part:} The fundamental step can also be achieved by a Corrugation Process as explained in section~\ref{Nash-Kuiper-historique}. We now assume that $f$ is a totally real map. We put $\pi=\sum_{i=1}^mc_i\varphi^i$, so that $\ell=d\pi$ and $u$ is any vector field such that $\ell(u)=1$. Let 
$$P_x := df(\ker d\pi_x)+Jdf(\ker d\pi_x)\subset T_{f(x)}W\;\mbox{ and }\; \textbf{t}(x):=\frac{[df(u_x)]^{P^{\perp}_x}}{\| [df(u_x)]^{P^{\perp}_x}\|_h}$$
where $[v]^{P_x^{\perp}}$ denotes the $P_x^{\perp}$ component of any vector $v$. Observe that $\textbf{t}(x)$ is a unit vector normal to $P_x$. As $JP_x=P_x$, $\textbf{n}(x) := J\textbf{t}(x)$ is normal to $P_x+\R\textbf{t}$(x) and thus to $df(T_xM)$. Since $\im_{TR}$ is a Kuiper relation with respect to $c$, the corrugation
\begin{eqnarray*}
CP_{\hgamma}(f,\pi, N) = \exp_{f}\left(\frac{r}{N}K_c(\alpha,N\pi)\textbf{t} +\frac{r}{N}K_s(\alpha,N\pi) \textbf{n}\right)
\end{eqnarray*}
produces the requested $\epsilon$-isometric totally real map $f_{\epsilon}:(U,\mu) \rightarrow (W,h,J)$ if $N$ is large enough. \\

\noi
It remains to show that the limit map $f_{\infty}$ of the sequence of totally real maps $(f_{k,i})$ generated by the above Corrugation Process is totally real as well. To do so, we consider the notion of $J$-density:  given a map $f:M\rightarrow (W,J,h)$ its $J$-density is the map $\kappa(f):M\rightarrow \R_{\geq 0}$ defined by
$$\kappa(f)(x)=\sqrt{|vol_{W} (df(e_1),\ldots,df(e_m),Jdf(e_1),\ldots,Jdf(e_m))|} $$ 
where $vol_W$ is the volume form of $W$ and $(e_1,\ldots,e_m)$ is any orthonormal basis of $T_xM$ for the pull back metric $(f^*h)_x$ (see \cite{Borrelli98}). Obviously, $\kappa(f)(x)$ does not depend on the chosen orthonormal basis and $df(T_xM)$ is totally real in $T_{f(x)}W$ if and only if $\kappa(f)(x)>0.$ It is Lagrangian if $\kappa(f)(x)=1$. Our goal is to show that $\kappa(f_{\infty})>0.$

\noi
\begin{lemme}\label{lemme_Jdensity1} If the vector field $u$ is chosen to be $f^*h$-orthogonal to $\ker d\pi$ on every point $x\in U$ then the $J$-density of the map $f_{\epsilon}=CP_{\hgamma}(f,\pi, N)$ satisfies $$\kappa(f_{\epsilon})\geq\frac{1}{\sqrt{\mu(u^*,u^*)}}\kappa(f)+O(1/N)$$
where $u^*:=u/\|u\|_{f^*h}$ is the normalized vector $u$ for the metric $f^*h.$
\end{lemme}

\begin{proof}
 of Lemma~\ref{lemme_Jdensity1}. Let $(e_1,\ldots,e_m)$ be a local $f^*h$-orthonormal basis of $TM$ over $U\subset M$ such that $Span(e_1,\ldots,e_{m-1}):=\ker \ell$. We choose the vector field $u$ to be $f^*h$-orthogonal to $\ker \ell$ and such that $\ell(u)=1.$ Observe that $e_m=\pm u/\|df(u)\|_h=\pm u^*.$ Let $\widetilde{e_m}=\frac{e_m}{\sqrt{\mu(e_m,e_m)}}=\pm \frac{u}{\sqrt{\mu(u,u)}}.$ Since
$$f_{\epsilon}^*h=\mu+O(1/N)\mbox{ with } \mu=f^*h+\rho\ell\otimes\ell+O(1/N)$$ 
by Section~\ref{Nash-Kuiper-historique}, we deduce that $(e_1,\ldots,e_{m-1},\widetilde{e_m})$ is $\mu$-orthogonal and thus approximatively  $f_{\epsilon}^*h$-orthonormal. In particular
\begin{eqnarray*}
\kappa(f_{\epsilon})^2 & =  & \displaystyle |vol_{W}	(df_{\epsilon}(e_1),\ldots,df_{\epsilon}(\widetilde e_m),Jdf_{\epsilon}(e_1),\ldots,Jdf_{\epsilon}(\widetilde e_m))|+O(1/N)\\
& =  & \displaystyle \frac{1}{\mu(e_m,e_m)}\displaystyle |vol_{W}	(df_{\epsilon}(e_1),\ldots,df_{\epsilon}(e_m),Jdf_{\epsilon}(e_1),\ldots,Jdf_{\epsilon}(e_m))|+O(1/N)
\end{eqnarray*}
The volume form $vol_W:\Lambda^{2m}(TW)\rightarrow \R$ is a $C^{\infty}$ function and from
$$d_{TW}(df_{\epsilon}(u),\gamma(\cdot,N\pi))=O(1/N)\;\;\mbox{ and }\;\;d_{TW }(df_{\epsilon}(e_j),df(e_j))=O(1/N) $$
for $j\in\{1,\ldots,m-1\},$ we deduce
\begin{eqnarray*}
\begin{array}{lll}
\kappa(f_{\epsilon})^2 & = & \displaystyle \frac{1}{\mu(u,u)}|vol_{W}	(df(e_1),\ldots,df(e_{m-1}),\gamma,Jdf(e_1),\ldots,Jdf(e_{m-1}),J\gamma)|\\
& & \hspace*{15mm}+O(1/N).
\end{array}
\end{eqnarray*}
Here $\gamma$ stands for $\gamma(\cdot,N\pi).$ Recalling that  
$$\gamma(\cdot,Nt)=r\cos(\alpha\cos(2\pi N t))\textbf{t} + r \sin(\alpha \cos(2\pi Nt))J\textbf{t} + [df(u)]^P$$
and replacing in the above expression we obtain
$$\kappa(f_{\epsilon})^2 = \displaystyle \frac{r^2}{\mu(u,u)}|vol_{W}	(df(e_1),\ldots,df(e_{m-1}),\textbf{t},Jdf(e_1),\ldots,Jdf(e_{m-1}),J\textbf{t})|
+O(1/N)$$
As
$$\textbf{t}=\frac{[df(u)]^{P^{\perp}}}{\| [df(u)]^{P^{\perp}}\|_h} = \pm \frac{[df(e_m)]^{P^{\perp}}}{\| [df(e_m)]^{P^{\perp}}\|_h}$$
we have
$$\kappa(f_{\epsilon})^2 = \frac{\kappa(f)^2}{\mu(u,u)} \frac{r^2}{\|[df(e_m)]^{P^{\perp}} \|_h^2} + O(1/N) = \frac{\kappa(f)^2}{\mu(e_m,e_m)} \frac{r^2}{\| [df(u)]^{\perp} \|_h^2}+O(1/N).$$
We then observe that
$$r^2=\mu(u,u)-\|[df(u)]^P\|_h^2=\|df(u)\|_h^2+\rho-\|[df(u)]^P\|_h^2 = \rho + \| [df(u)^{P^{\perp}}] \|_h^2$$
to obtain
$$\kappa(f_{\epsilon})\geq\frac{1}{\sqrt{\mu(e_m,e_m)}}\kappa(f)+O(1/N).$$
\end{proof}

\noi
In the sequel, it is invariably decided to choose the vector field $u$ to be $f^*h$-orthogonal to $\ker d\pi= \ker \ell$. To plug Lemma~\ref{lemme_Jdensity1} in the body of the work, we need to rephrase it. This is the purpose of the next lemma.

\begin{lemme}\label{lemme_Jdensity2} We have
$$\kappa(f_{k,i+1})\geq\frac{1}{1+(\delta_{k+1}-\delta_k)\|\Delta\|_{f_0^*h}}\kappa(f_{k,i})+O(1/N_{k,i})$$
where $\displaystyle\|\Delta\|_{f_0^*h}=\sup_{\{v_x\in TM\; |\; \|v\|_{f_0^*h}=1\}}\Delta(v_x,v_x).$
\end{lemme}

\begin{proof}
 of Lemma~\ref{lemme_Jdensity2}. Let 
$\mu_{k,i}:=f_{k,i}^*h+\rho_{k,i}\ell_{k,i}\otimes\ell_{k,i}.$ We have
$$\begin{array}{lll}
0\leq\mu_{k,i}(u_{k,i}^*,u_{k,i}^*)-1 & = & \mu_{k,i}(u_{k,i}^*,u_{k,i}^*)-f_{k,i}^*h(u_{k,i}^*,u_{k,i}^*)\\
               & \leq & g_{k+1}(u_{k,i}^*,u_{k,i}^*)-g_k(u_{k,i}^*,u_{k,i}^*)\\
               & \leq & (\delta_{k+1}-\delta_k)\Delta(u_{k,i}^*,u_{k,i}^*)
\end{array}$$
We observe that $1=\|u_{k,i}^*\|_{f_{k,i}^*h}\geq \|u_{k,i}^*\|_{f_{0}^*h}$ because $f_{k,i}^*h\geq f_{0}^*h$, thus
$$\Delta(u_{k,i}^*,u_{k,i}^*)\leq  \|\Delta\|_{f_0^*h}$$
and
$$\mu_{k,i}(u_{k,i}^*,u_{k,i}^*)\leq 1+(\delta_{k+1}-\delta_k)\|\Delta\|_{f_0^*h}.$$
Lemma~\ref{lemme_Jdensity2} is now a straightforward consequence of Lemma~\ref{lemme_Jdensity1}.
\end{proof}

\noi
 So, if the corrugation numbers $(N_{k,i})$ and the sequence $(\delta_k)_k$ are conveniently chosen, we can insure that $\kappa(f_{\infty})\geq C\kappa(f_0)$ for some $0<C<1$ which shows that $f_{\infty}$ is totally real.

\subsection{Gauss and Maslov maps}\label{subsection_gauss_maslov}

\noi
\textbf{Totally real Grassmannian and Maslov map.--}
We denote by $TR(m)$ the Grassmannian of totally real $m$ planes of $\C^m$. Given a totally real $m$-plane $\Pi_0$ of $\C^m$, this Grassmannian is identified with the homogeneous space $GL(m,\C)/GL(m,\R)$ via the map $i_{\Pi_0}:\Pi\mapsto [L]$ where $L$ is any $\C$-linear map such that $L(\Pi_0)=\Pi$. This homogeneous space admits a fibration $\phi:GL(m,\C)/GL(m,\R)\rightarrow \Sph^1$ given by 
$$[L]\longmapsto \frac{\det^2L}{|\det^2L|}.$$
Let $p:TR(W)\rightarrow W$ be the totally real Grassman bundle of $(W,J)$ (where $\dim_{\R} W =2m$). Any choice of a totally real $m$-plane $\Pi_0(y)\subset T_yW$ induces an identification between the fiber $p^{-1}(y)$ and $GL(m,\C)/GL(m,\R).$ Thus any local section $\Pi_0:V\subset W\rightarrow TR(W)$ allows to define a map $\phi\circ i_{\Pi_0}:p^{-1}(V)\rightarrow \Sph^1$. Given such a section $\Pi_0$, the Gauss map 
$$G_f:M\longrightarrow TR(W),\;\;\; x\longmapsto (f(x),df(T_xM))$$
of any totally real embedding $f: M\to (W,J)$  such that $f(M)\subset V$ induces a map $\mathfrak{m}(\Pi_0,f):=\phi\circ i_{\Pi_0}\circ G_f:M\rightarrow \Sph^1$ that we call the {\it Maslov map}. Observe that a local section $\Pi_0$ can be constructed from a totally real embedding $f_0: M\to (W,J)$: $V$ is a tubular neighborhood of $f_0(M)$ and $\Pi_0$ is any extension of $f_0(x)\mapsto df_0(T_xM)$. In that case, if $f(M)\subset V$, we denote by $\mathfrak{m}(f_0,f)$ the corresponding Maslov map.\\

\noi
\textbf{Maslov map of $f_{\infty}$}.-- In the above proof, every map $f_{k,j}$ as well as $f_{\infty}$ have images lying inside an $\epsilon$-tubular neighborhood $V(\epsilon)$ of $f_0(M)$. If $\epsilon$ is small enough, this neighborhood retracts by deformation on $f_0(M)$ and it can be used to construct a local section $\Pi_0$ extending the one induced by $f_0$. In that case, for every $(k,j)$ we write
$$x\mapsto\mathfrak{m}(f_0, f_{k,j})(x) = e^{2i\vartheta_{k,j}(x)}\mathfrak{m}(f_0, f_{k,j-1})(x)$$
where $\vartheta_{k,j}$ is some angle function. If $k$ is large enough we choose this angle function such that $\vartheta_{k,j}(x)\in\;]-\frac{\pi}{2},\frac{\pi}{2}[$ for all $x\in M$ (this is always possible by the convergence of $(f_{k,j})$). We define inductively a sequence of maps $\mathcal{W}_{k} : M\rightarrow\R$ by $\mathcal{W}_0=0$, $\mathcal{W}_{k+1}:=\mathcal{W}_{k}+2\vartheta_{k}$ with $\vartheta_k := \sum_{j\in I(k)}\vartheta_{k,j}$. Since the $f_{k}$'s are $C^1$ converging toward $f_{\infty}$, the maps $\mathcal{W}_{k}$ also $C^0$ converge toward a map $\mathcal{W}_{\infty}$ such that 
$$\mathfrak{m}(f_0,f_{\infty}) = e^{i\mathcal{W}_{\infty}}.$$

\noi
The following proposition matches with Proposition~\ref{prop_maslov_intro} for $\ell=(k,j)$.

\begin{prop}
\label{prop_maslov}
Let  $\mathfrak{m}(f_0,f_{\infty})=e^{i\mathcal{W}_{\infty}}:M\rightarrow\Sph^1$ be the Maslov map of $f_{\infty}$ and $\mathcal{W}_{\infty}=2\sum_k\vartheta_k$ be the Maslov argument defined above. Then if $k$ is large enough
$$\vartheta_{k}=\theta_{k}+\sum_{j\in I(k)} O\left( \frac{1}{N_{k,j}} \right)
\mbox{ where  }
\theta_{k}:=\sum_{j\in I(k)}\alpha_{k,j}\cos(2\pi N_{k,j}\pi_{k,j})$$
(if $x\in M$ is not in the domain of $\pi_{k,j}$ it is understood that the corresponding term is zero)
\end{prop}

\noi
The proof of this proposition is a straightforward consequence of the following lemma.

\begin{lemme}\label{lemme_Maslov_angle} Let $f:U\rightarrow W$ be a totally real map. Let $\pi :U\rightarrow \R$ be a submersion, $u:U\rightarrow TU$ a vector field chosen to be $f^*h$-orthogonal to $\ker d\pi$ on every point $x\in U$ and such that $d\pi(u)=1.$ Then, for all $x\in U$:
$$\mathfrak{m}(\Pi_0,f_{\epsilon})(x)=e^{2i(\alpha(x)\cos(2\pi N \pi(x))+O(1/N))}\mathfrak{m}(\Pi_0,f)(x)$$
where $f_{\epsilon}=CP_{\hgamma}(f,\pi, N).$
\end{lemme}

\begin{proof} of the lemma. Let $(e_1,\ldots,e_m)$ be a local basis of $TM$ over $U\subset M$ such that $Span(e_1,\ldots,e_{m-1}):=\ker d\pi$ and $e_m=u$. Let $(\epsilon_1,\ldots,\epsilon_m)$ be a local basis of the $m$-plane field $\Pi_0$. We consider the square of the determinant map $z(f):U\rightarrow\C$:
$$z(f):=\mbox{det}_{\C}^2(df(e_1),\ldots,df(e_m))$$
where $df(e_1),\ldots,df(e_m)$ are seen as complex vectors with complex coordinates relative to the basis $(\epsilon_1,\ldots,\epsilon_m)$. 
Since $df(T_xM)$ is totally real for every $x$, the complex number $z(f)(x)$ never vanishes. It is readily seen that its argument is $\mathfrak{m}(\Pi_0,f)(x).$ Similarly, the argument of
$$z(f_{\epsilon})=\mbox{det}_{\C}^2(df_{\epsilon}(e_1),\ldots,df_{\epsilon}(e_m))$$
is $\mathfrak{m}(\Pi_0,f_{\epsilon})$. From the smoothness of $\det_{\C}^2$ and from
$$d_{TW}(df_{\epsilon}(u),\gamma(\cdot,N\pi))=O(1/N)\;\;\mbox{ and }\;\;d_{TW }(df_{\epsilon}(e_j),df(e_j))=O(1/N) $$
for $j\in\{1,\ldots,m-1\},$ we deduce
$$z(f_{\epsilon}) = \mbox{det}_{\C}^2(df(e_1),\ldots,df(e_{m-1}),\gamma(\cdot,N\pi))+O(1/N).$$
As
$$\gamma(\cdot,t)=r\, e^{i\alpha\cos(2\pi N t)}\textbf{t} + [df(u)]^P \,\, \mbox{ and } \,\, \textbf{t}= \frac{[df(u)]^{P^{\perp}}}{\|[df(u)]^{P^{\perp}}\|}$$
we have 
$$z(f_{\epsilon})=\frac{r^2}{ \|[df(u)]^{P^{\perp}}\|^2}e^{2i\alpha\cos(2\pi N \pi(\cdot))}z(f)+O(1/N).$$
\end{proof}

\subsection{Self similarities and isometric maps}

We put in perspective the result obtained for totally real maps with the ones obtained for the isometric embeddings of the square flat torus and the reduced sphere in \cite{ensaios,BBDLRT}. In these papers, the differential of the isometric embedding is expressed by an infinite product of \emph{corrugations matrices} and a self similarity property ensues from a \emph{Corrugation Theorem} (Theorem 21 in \cite{ensaios}). We put into light in this subsection that we have similar infinite product of rotations and self similarities behavior.\\

\noi
\textbf{Corrugation matrices.--} We assume $n=m+1$. Let $U\simeq [0,1]^m$ be a chart of a $m$-dimensional oriented manifold  and let $f_{k,j}:U\to \E^{m+1}$ be a sequence obtained iteratively by the Nash process (see Subsection \ref{subsection:proofthm2}). 
We  recall the definition of the corrugation matrices.  We denote by $\ell_{k,j}$ the successive linear forms. We are going to build two basis of $\E^{m+1}$ to express the corrugation matrices.
We denote by $u_{k,j+1}$ any  vector field  such that $\ell_{k,j+1}(u_{k,j+1})=1$ and by $\t_{k,j}$ the normalisation of $df_{k,j}(u_{k,j+1})$. Let $(V_{k,j+1}^1,\cdots,V_{k,j+1}^{m-1})$ be any local basis of $\mathrm{Ker}\ \ell_{k,j+1}$. We put 
$$
\n_{k,j}:= \frac{\t_{k,j} \wedge df_{k,j}(V_{k,j+1}^1)\wedge \cdots \wedge df_{k,j}(V_{k,j+1}^{m-1})}{\| \t_{k,j} \wedge df_{k,j}(V_{k,j+1}^1)\wedge \cdots \wedge df_{k,j}(V_{k,j+1}^{m-1})\|}.
$$
We denote by $v_{k,j}^1,\cdots,v_{k,j}^{m-1}$ the Gram-Schmidt orthonormalisation of $df_{k,j}(V_{k,j}^1)$, $\cdots$ ,$df_{k,j}(V_{k,j}^{m-1})$ and we put
$$
v_{k,j}^\perp:=v_{k,j}^1 \wedge \cdots v_{k,j}^{m-1} \wedge \n_{k,j}.
$$
Let $p\in U$. Observe that $\mathcal{B}_{k,j}(p):=(v_{k,j}^\perp, v_{k,j}^1, \cdots, v_{k,j}^{m-1},\n_{k,j})(p)$ is an orthonormal basis of $\E^{m+1}$. 
We also introduce $m-1$ vectors $v_{k,j}^{1+},\cdots,v_{k,j}^{(m-1)+}$ as the Gram-Schmidt orthonormalisation of $df_{k,j}(V_{k,j+1}^1)$,$\cdots$,$df_{k,j}(V_{k,j+1}^{m-1})$ and we define a second orthonormal basis $\mathcal{B}_{k,j}^+(p):=(\t_{k,j}^\perp, v_{k,j}^{1+}, \cdots, v_{k,j}^{(m-1)+},\n_{k,j})(p)$. 
We denote by $\Rcal_{k,j}(p)$ the rotation matrix that maps $\mathcal{B}_{k,j}(p)$ to $\mathcal{B}_{k,j}^+(p)$ and by $\Lcal_{k,j+1}(p,N_{k,j+1})$ the rotation matrix that maps $\mathcal{B}_{k,j}^+(p)$ to $\mathcal{B}_{k,j+1}(p)$. The \emph{corrugation matrix} is defined by the product 
$$
\mathcal{M}_{k,j+1}(p,N_{k,j+1}):=\Lcal_{k,j+1}(p,N_{k,j+1}) \Rcal_{k,j}(p)
$$ 
and  maps $\mathcal{B}_{k,j}(p)$ to $\mathcal{B}_{k,j+1}(p)$. 

\bigskip

\noi
\textbf{Corrugation theorem.--} Corrugation matrices encode both the differential and the Gauss map of the limit embedding, via the following infinite product:
$$
\left(
\begin{array}{c}
v_{\infty}^\perp\\
v_{\infty}^1\\ 
\vdots\\
v_{\infty}^{m-1}\\
\n_{\infty}
\end{array}
\right)(p)
= \left( \prod_{k,j}\mathcal{M}_{k,j+1}(p,N_{k,j+1}) \right)
\left(
\begin{array}{c}
v_{0}^\perp\\
v_{0}^1\\ 
\vdots\\
v_{0}^{m-1}\\
\n_{0}
\end{array}
\right)(p)
$$
In their construction of an isometric embedding of the flat torus, the authors succeed in reducing the number of the direction of corrugation to three. As a consequence, they could prove that the matrices $\Rcal_{k,j}$ converge toward three constant matrices (depending on the value of $j$) when $k$ goes to infinity. Regarding the matrices $\Lcal_{k,j}$ they show that they are equal to a simple rotation of angle $\alpha_{k,j+1}(p)\cos(2\pi N_{k,j+1}\pi_{k,j+1}(p))$ modulo $O(1/N_{k,j+1})$. Therefore, up to the three constant matrices, the infinite product of the $\Mcal_{k,j}$ is similar to a product of rotations whose angles oscillate with increasing frequencies. Here, we observe that if the property on $\Rcal_{k,j}$ is specific to their construction, the asymptotic behavior of $\Lcal_{k,j}$ still holds in a general context. 

\begin{prop}\label{prop:fin} Let $f_{k,j}:U\simeq [0,1]^m\to \E^{m+1}$ be a sequence obtained iteratively by the Nash process. If $f_{k,j+1}$ is obtained from $f_{k,j}$ by the corrugation process of Proposition~\ref{kuiper_formula_is}, then for every $p \in U$, we have
$$
\Lcal_{k,j+1}(p,N_{k,j+1}) = 
\left(
\begin{array}{ccccc}
\cos \theta_{k,j+1} & 0   & \sin \theta_{k,j+1}\\
0& Id & 0\\
-\sin \theta_{k,j+1} & 0 & \cos \theta_{k,j+1}
\end{array}
\right)
+ O\left(\frac{1}{N_{k,j+1}}\right),
$$
where $\theta_{k,j+1}:=\alpha_{k,j+1}(p)\cos(2\pi N_{k,j+1}\pi_{k,j+1}(p))$.
\end{prop}

\begin{proof}
The proof is a straightforward adaptation of Lemma 20 from \cite{ensaios}. 
\end{proof}

\noi
\textbf{The totally real case.--} In the totally real case, the above approach can be adapted as follows. Let $V^1$, $\cdots$, $V^m$ be a local basis and $f_{k,j}$ be a totally real map. We put $v_{k,j}^i=df_{k,j}(V^i)$ and observe that $\mathcal{B}_{k,j}=(v_{k,j}^1, \cdots, v_{k,j}^m,Jv_{k,j}^1, \cdots, Jv_{k,j}^m)$ is a basis of $\E^{2m}$. The differential of the limit embedding $f_\infty$ is determined by
$$
\left(
\begin{array}{c}
v_{\infty}^1\\ 
\vdots\\
v_{\infty}^{m}\\
Jv_{\infty}^1\\ 
\vdots\\
Jv_{\infty}^{m}\\
\end{array}
\right)
= \left(\prod_{k,j}\mathcal{M}_{k,j+1}\right)
\left(
\begin{array}{c}
v_{0}^1\\
\vdots\\
v_{0}^{m}\\
Jv_{0}^1\\ 
\vdots\\
Jv_{0}^{m}\\
\end{array}
\right),
$$
where $\Mcal_{k,j+1}$ is the matrix that maps the basis $\mathcal{B}_{k,j}$ to the basis $\mathcal{B}_{k,j+1}$. 
Each $\Mcal_{k,j}$ is a $2m\times 2m$ real matrix which commutes with $J$. We denote by $\mathcal{M}_{k,j}^{\C} \in GL(m,\C)$ its complexification. 
The Maslov map $\mathfrak{m}(f_0,f_\infty)$ is the square of the determinant of the infinite product of the $\mathcal{M}_{k,j}^{\C}$'s divided by its module (so that its image lies in $\mathrm{U}(1)$). Observe that this Maslov map is not affected by the rotations in the tangent spaces encoded by the matrices $\Rcal_{k,j}$. As stated in Proposition~\ref{prop_maslov}, this fact allows to obtain an analogy with a Weierstrass function even if the directions of the corrugations are not under control.

\bibliographystyle{alpha}
\bibliography{bib_art01}

\listoffigures

\end{document}